\documentclass[twocolumn,10pt,letterpaper]{asme2ej}

\usepackage{epsfig} % for postscript graphics files
\usepackage{mathptmx} % assumes new font selection scheme installed
\usepackage{times} % assumes new font selection scheme installed
\usepackage{amsmath} % assumes amsmath package installed
\usepackage{amssymb}  % assumes amsmath package installed
\usepackage{tabularx}
\usepackage{algorithm}
\usepackage{algorithmic}
\usepackage{ulem}
\usepackage{dsfont}
\usepackage{multirow}

\usepackage[noadjust]{cite}

\usepackage[dvipsnames,usenames]{color}
\usepackage[colorlinks,%
linkcolor=BrickRed,%
filecolor=BrickRed,%
citecolor=RoyalPurple,%
]{hyperref}
\usepackage{caption}
\usepackage{subfigure}

\newcounter{rmnum}
\newenvironment{romannum}{\begin{list}{{\upshape (\roman{rmnum})}}{\usecounter{rmnum}
\setlength{\leftmargin}{14pt}
\setlength{\rightmargin}{8pt}
\setlength{\itemsep}{2pt}
\setlength{\itemindent}{-1pt}
}}{\end{list}}

\usepackage{color}

\usepackage{tabularx}

\def\Re{\mathbb{R}}

\def\K{{\sf K}}

\def\argmin{\mathop{\text{\rm arg\,min}}}
\def\ind{\text{\rm\large 1}}
\def\varble{\,\cdot\,}
\def\varble{\,\cdot\,}

\def\Fig#1{Fig.~\ref{#1}}

\def\notes#1{\marginpar{\tiny #1}\typeout{Notes!
Notes!
Notes!
}}
\renewcommand{\notes}[1]{\typeout{notes!}}

\def\ind{\bbbone}

\def\Re{\field{R}}
\def\k{{\sf K}}

\def\clZ{{\cal Z}}

\def\E{{\sf E}}
\def\Expect{{\sf E}}

\def\Prob{{\sf P}}

\def\Expect{{\sf E}}

\def\Fig#1{Fig.~\ref{#1}}

\newcommand{\expect}{ {\sf E} }

\def\clZ{{\cal Z}}
\def\varble{\,\cdot\,}

\newtheorem{theorem}{Theorem}
\newtheorem{example}{Example}

\newtheorem{remark}{Remark}

\def\beq{\begin{eqnarray}} 
\def\bc{\begin{center}} 
\def\be{\begin{enumerate}}
\def\bi{\begin{itemize}} 
\def\bs{\begin{small}}
\def\bS{\begin{slide}}
\def\ec{\end{center}} 
\def\ee{\end{enumerate}}
\def\ei{\end{itemize}}
\def\es{\end{small}}
\def\eS{\end{slide}}
\def\eeq{\end{eqnarray}}

\def\hah{\hat{h}}

\def\hah{\hat{h}}

\newcommand{\ud}{\,\mathrm{d}}

\def\Re{\mathbb{R}}
\def\E{{\sf E}}

\def\argmin{\mathop{\text{\rm arg\,min}}}
\def\ind{\text{\rm\large 1}}
\def\varble{\,\cdot\,}

\def\Expect{{\sf E}}

\def\clZ{{\cal Z}}

\def\hah{\hat{h}}

\renewcommand{\Re}{\mathbb{R}}

\def\Prob{{\sf P}}

\newcommand{\X}{X}

\newcommand{\kepsN}{{k}_{\epsilon}^{(N)}}
\newcommand{\nepsN}{{n}_{\epsilon}^{(N)}}

\newcommand{\phiepsN}{{\phi}^{(N)}_\epsilon}

\newcommand{\phiM}{\phi^{(M)}}

\newcommand{\TepsN}{{T}^{(N)}_\epsilon}

\newcommand{\geps}{{g}_{\epsilon}}
\newcommand{\pr}{\rho}

\title{Kalman Filter and its Modern Extensions for the
  Continuous-time Nonlinear Filtering Problem}

\author{Amirhossein Taghvaei, 
    \affiliation{
	Department of Mechanical Science and Engineering\\
	University of Illinois at Urbana-Champaign\\
	Email: taghvae2@illinois.edu
    }	
}

\author{Jana de Wiljes,
    \affiliation{
    Institut f\"ur Mathematik \\ 
    Universit\"at Potsdam\\
    Email: wiljes@uni-potsdam.de
    }	
}
\author{Prashant G. Mehta, 
    \affiliation{
	Department of Mechanical Science and Engineering\\
	University of Illinois at Urbana-Champaign\\
	Email: mehtapg@illinois.edu
    }	
}
\author{Sebastian Reich, 
    \affiliation{
	Institut f\"ur Mathematik\\ 
	Universit\"at Potsdam and\\
	Department of Mathematics and Statistics \\
	University of Reading\\
	Email: sereich@uni-potsdam.de
    }	
}

\begin{document}
\normalem
\maketitle
\begin{abstract}
This paper is concerned with the filtering problem in
continuous-time.  Three algorithmic solution approaches for this
problem are reviewed:  (i) the classical Kalman-Bucy filter which provides
an exact solution for the linear Gaussian problem, (ii) the ensemble
Kalman-Bucy filter (EnKBF) which is an approximate filter and represents an
extension of the Kalman-Bucy filter to nonlinear problems, and (iii) the
feedback particle filter (FPF) which represents an extension of the EnKBF and
furthermore provides for an consistent solution in the
general nonlinear, non-Gaussian case.  The common feature of the three
algorithms is the gain times error formula to implement the
update step (to account for conditioning due to the observations) in the filter.  
In contrast to the commonly used sequential Monte Carlo methods, the EnKBF and FPF avoid the resampling of the 
particles in the importance sampling update step.  Moreover, the
feedback control structure
provides for error correction potentially leading to smaller
simulation variance and improved stability properties.  The paper
also discusses the issue of non-uniqueness of the filter update
formula and  formulates a novel approximation algorithm based on ideas
from optimal transport and coupling of measures.  Performance of this
and other algorithms is illustrated for a numerical example.

\end{abstract}
%\tableofcontents

\section{Introduction}
\label{sec:intro}

\begin{table*}[t]
\centering
\begin{tabular}{|c|c|c|}
\hline
KBF & Kalman-Bucy Filter & Equations~\eqref{eq:mean_KF}-\eqref{eq:var_KF}\\ \hline
EKBF & Extended Kalman-Bucy Filter & Equations~\eqref{eq:mean-EKF}-\eqref{eq:var-EKF} \\ \hline
\multirow{2}{*}{EnKBF} & (Stochastic) Ensemble Kalman-Bucy Filter & Equation~\eqref{Stochastic EnKBF} \\ & (Deterministic) Ensemble Kalman-Bucy Filter & Equation~\eqref{detEnKBF}\\ \hline
FPF & Feedback Particle Filter & Equation~\eqref{eqn:particle_filter_nonlin_intro}
\\\hline
\end{tabular}
\caption{Nomenclature for the continuous-time filtering algorithms}  
\label{tab:nomen}
\end{table*}

Since the pioneering work of Kalman from the 1960s, sequential state
estimation has been extended to application areas far beyond its
original aims such as numerical weather prediction \cite{sr:kalnay}
and oil reservoir exploration (history matching)
\cite{sr:Oliver2008}. These developments have been made possible by
clever combination of Monte Carlo techniques with Kalman-like
techniques for assimilating observations into the underlying dynamical
models.  The most prominent of these algorithms are the ensemble
Kalman filter (EnKF), the randomized maximum likelihood (RML) method
and the unscented Kalman filter (UKF) invented independently by
several research groups
\cite{sr:kitanidis95,jdw:BurgersLeeuwenEvensen1998,sr:houtekamer01,sr:Julier97anew}
in the 1990s.  The EnKF in particular can be viewed as a cleverly
designed random dynamical system of interacting particles which is
able to approximate the exact solution with a relative small number of
particles. This interacting particle perspective has led to many new
filter algorithms in recent years which go beyond the inherent
Gaussian approximation of an EnKF  during the data assimilation
(update) step
\cite{jdw:ReichCotter2015}.

In this paper, we review the interacting particle perspective in the
context of the continuous-time filtering problems and demonstrate its
close relation to Kalman's and Bucy's original feedback control structure of the
data assimilation step.  More specifically, we highlight the feedback 
control structure of three classes of algorithms  for approximating
the posterior distribution: (i) the classical Kalman-Bucy filter which
provides an exact solution for the linear Gaussian problem, (ii) the
ensemble Kalman-Bucy filter (EnKBF) which is an approximate filter and
represents an extension of the Kalman-Bucy filter to nonlinear
problems, and (iii) the feedback particle filter (FPF) which
represents an extension of the EnKBF and furthermore provides for a
consistent solution 
of the general nonlinear, non-Gaussian problem.

A closely related goal is to provide comparison between these
algorithms.  A common feature of the three algorithms is the gain
times error formula to implement the update step in the filter.   
% It is shown that both the EnKBF and the FPF algorithm are
% consistent with the Kalman-Bucy filter in the linear Gaussian setting.  
The
difference is that while the Kalman-Bucy filter is an exact algorithm, the
two particle-based algorithms are approximate with error decreasing to
zero as the number of particles $N$ increases to infinity.
Algorithms with this property are said to be consistent.    

In the class of interacting particle algorithms discussed, the FPF
represents the most general solution to the nonlinear non-Gaussian
filtering problem.  The challenge with implementing the FPF lie in
approximation of the ``gain function'' used in the update step.  The gain function
equals the Kalman gain in the linear Gaussian setting and must be
numerically approximated in the general setting.  One particular
closed-form approximation is the 
constant gain approximation.  In this case, the FPF is shown to reduce
to the EnKBF algorithm.  The  EnKBF naturally extends to nonlinear
dynamical systems and its discrete-time versions have  become very
popular in recent years with applications to, for example,
atmosphere-ocean dynamics and oil reservoir exploration.  In
the discrete-time setting, development and application of closely related
particle flow algorithms has also been a subject of recent
interest, e.g.,~\cite{daum10,daum2017generalized,ColemanPosterior,MarzoukBayesian,2015arXiv150908787H,yang_discrete}.

The outline of the remainder of this paper is as follows: The continuous-time 
filtering problem and the classic Kalman-Bucy filter are summarized in Sections 
\ref{sec:problem} and \ref{sec:KBF}, respectively. The Kalman-Bucy filter is then
put into the context of interacting particle systems in the form of
the EnKBF in Section \ref{sec:EnKF}.
Section \ref{sec:FPF} together with Appendices \ref{sec:theory} and \ref{sec:lp_justification} provides
a consistent definition of the FPF and a discussion of alternative
approximation techniques which lead to consistent approximations to
the filtering problem as the number of particles, $N$, goes to
infinity.  It is shown that the EnKBF can be viewed
as an approximation to the FPF.  Four
algorithmic approaches to gain function approximation are described
and their relationship discussed. The performance of the four algorithms is numerically studied and compared for an example problem in Section~\ref{sec:numerics}. The paper concludes with discussion
of some open problems in Section~\ref{sec:conc}.  The nomenclature for
the filtering algorithms described in this paper appears in
Table~\ref{tab:nomen}.

\section{Problem Statement}
\label{sec:problem}

In the continuous-time setting, the model for nonlinear filtering problem is
described by the nonlinear stochastic differential equations (sdes):
\begin{subequations}
\begin{align}
\text{Signal:} \quad \quad 
\ud X_t &= a(X_t) \ud t + \sigma(X_t) \ud B_t,\quad\quad X_0 \sim p_0^*
\label{eqn:Signal_Process}
\\
\text{Observation:} \quad \quad \ud Z_t &= h(X_t)\ud t + \ud W_t
\label{eqn:Obs_Process}
\end{align}
\end{subequations}
where $X_t\in\Re^d$ is the (hidden) state at time $t$, the initial
condition $X_0$ is sampled from a given prior density $p_0^*$, $Z_t \in\Re^m$ is the
observation or the measurement vector, and $\{B_t\}$, $\{W_t\}$ are two mutually
independent Wiener processes taking values in $\Re^d$ and $\Re^m$. The
mappings $a(\cdot): \Re^d \rightarrow \Re^d$, $h(\cdot): \Re^d
\rightarrow \Re^m$ and $\sigma(\cdot):\Re^d \rightarrow \Re^{d \times
  d}$ are known $C^1$ functions.  
The covariance matrix of the observation
noise $\{W_t\}$ is assumed to be positive definite.
The function $h$ is a column vector whose $j$-th coordinate is denoted as $h_j$ (i.e.,
$h=(h_1,h_2,\hdots,h_m)^{\rm T}$).  
By scaling, it is assumed without loss of generality that the covariance
matrices associated with $\{B_t\}$, $\{W_t\}$ are identity matrices.
Unless otherwise noted, the stochastic differential equations (sde)
are expressed in It\^{o} form.  Table~\ref{tab:symbols-filter}
includes a list of symbols used in the continuous-time filtering
problem.

In applications, the continuous time filtering models are
often expressed as:
\begin{subequations}
\begin{align}
\frac{\ud X_t}{\ud t}  &= a(X_t)  + \sigma(X_t) \dot{B}_t
\label{eqn:Signal_Process_1}
\\
Y_t & := \frac{\ud Z_t}{\ud t} = h(X_t) + \dot{W}_t
\label{eqn:Obs_Process_1}
\end{align}
\end{subequations}
where $\dot{B}_t$ and $\dot{W}_t$ are mutually independent white noise
processes (Gaussian noise) and $Y_t \in\Re^m$ is the vector valued
observation at time $t$.  The sde-based model is preferred here
because of its mathematical rigor.  Any sde involving $Z_t$ is
converted into an ODE involving $Y_t$ by formally dividing the sde by
$\ud t$ and replacing $\frac{\ud Z_t}{\ud t}$ by $Y_t$ (See also
Remark~\ref{rem:remStrato}). 

\begin{table}[t]
\centering
\begin{tabular}{|c|c|c|}
\hline
Variable & Notation & Model\\ \hline
State & $X_t$ & Eq.~\eqref{eqn:Signal_Process} \\ 
Process noise  & $B_t$ & Wiener process \\ 
Measurement & $Z_t$ & Eq.~\eqref{eqn:Obs_Process}\\
Measurement noise & $W_t$ & Wiener process
\\\hline
\end{tabular}
\caption{Symbols for the continuous-time filtering problem}  
\label{tab:symbols-filter}
\end{table}

\medskip

The objective of filtering is to estimate the
posterior distribution of $X_t$ given the time history of observations
$\clZ_t :=
\sigma(Z_s:  0\le s \le t)$. The density of the posterior
distribution is denoted by $p^*$, so
that for any measurable set $A\subset \Re^d$,
\begin{equation}
\int_{x\in A} p^*(x,t)\, \ud x   = \Prob\{ X_t \in A\mid \clZ_t \}
\nonumber
\end{equation}

\medskip

One example of particular interest is when the mappings $a(x)$ and
$h(x)$ are linear, $\sigma(x)$ is a constant matrix that does not
depend upon $x$, and
the prior density $p_0^*$ is Gaussian.  The associated problem is
referred to as the linear Gaussian filtering problem. For this
problem, the posterior density is known to be Gaussian.  The resulting
filter is said to be finite-dimensional because the posterior is
completely described by finitely many statistics -- conditional mean
and variance in the linear Gaussian case. 

\medskip

For the general nonlinear non-Gaussian case, however, the filter is
infinite-dimensional because it defines the evolution, in  the space of
probability measures, of $\{p^*(\varble ,t) : t\ge 0\}$.   
The particle filter is a simulation-based algorithm to approximate the
posterior: The key step is the
construction of $N$ interacting stochastic processes $\{X^i_t : 1\le i \le N\}$:
The value $X^i_t \in \Re^d$ is the state for the $i$-th
particle at time $t$. For each time $t$, the empirical distribution
formed by the particle population is used to approximate the
posterior distribution.  Recall that this is  defined for any
measurable set $A\subset\Re^d$ by,
\begin{equation} \label{eqn:empirical}
p^{(N)}(A,t) = \frac{1}{N}\sum_{i=1}^N \ind\{ X^i_t\in A\} \,.\nonumber
%\label{e:piN}
\end{equation}
where $\ind\{x \in A\}$ is the indicator function (equal to $1$ if
$x\in A$ and $0$ otherwise).  The first interacting particle representation of the continuous-time filtering
problem can be found in \cite{crisan2007,crisan10}. The close connection of such interacting
particle formulations to the gain factor and innovation structure of the classic Kalman filter
has been made explicit starting with \cite{taoyang_cdc11,taoyang_TAC12} and has led to the FPF formulation 
considered in this paper.

\medskip

\noindent \textbf{Notation:} The density for a Gaussian random
variable with mean $m$ and variance $\Sigma$ is denoted as ${\cal
  N}(m,\Sigma)$.  For vectors $x,y\in\Re^d$, the dot product is
denoted as $x\cdot y$ and $|x|:=\sqrt{x\cdot x}$; $x^{\rm T}$ denotes
the transpose of the vector.  Similarly, for a matrix $\K$, $\K^{\rm
  T}$ denotes the matrix transpose. For two sequence
$\{a_n\}_{n=1}^\infty$ and $\{b_n\}_{n=1}^\infty$, the big $O$
notation $a_n=O(b_n)$ means $\exists \, n_0 \in \mathbb{N}$ and $c>0$ such that $|a_n| \leq c|b_n|$ for $n>n_0$.

\section{Kalman-Bucy Filter}
\label{sec:KBF}

\begin{table}[t]
\centering
\begin{tabular}{|c|c|c|}
\hline
Variable & Notn. \& Defn. & Model\\ \hline
Cond. mean & $\hat{X}_t=\Expect[X_t|\clZ_t]$ & Eq.~\eqref{eq:mean_KF} \\ 
Cond. var.  & $\Sigma_t=\expect[(X_t-\hat{X}_t) (X_t-\hat{X}_t)^T|\clZ_t]$ & Eq.~\eqref{eq:var_KF} \\ 
Kalman gain & $\K_t=\Sigma_tH_t^T$ & Eq.~\eqref{eq:Kalman_gain} 
\\\hline
\end{tabular}
\caption{Symbols for the Kalman filter}  
\label{tab:symbols-KF}
\end{table}

Consider the linear Gaussian problem: The mappings $a(x) = A\,x$ and
$h(x) = H\,x$ where $A$ and $H$ are $d \times d$ and $m \times d$
matrices; the process noise covariance $\sigma(x) = \sigma$, a
constant $d \times d$ matrix; and the prior density is Gaussian,
denoted as ${\cal N}(\hat{X}_0,\Sigma_0)$.

For this problem, the posterior density is known to be Gaussian,
denoted as ${\cal N}(\hat{X}_t,\Sigma_t)$, where $\hat{X}_t$ and
$\Sigma_t$ are the conditional mean and variance, i.e
$\hat{X}_t:=\expect[X_t|\clZ_t]$ and $\Sigma_t:=\expect[(X_t-\hat{X}_t) (X_t-\hat{X}_t)^T|\clZ_t]$.  Their evolution is
described by the finite-dimensional Kalman-Bucy filter:
\begin{subequations}
\begin{align}
 \ud \hat{X}_t  &= A \hat{X}_t \ud t + {\sf K}_t 
\Bigl(\ud
Z_t- H_t \hat{X}_t \ud t \Bigr)
\label{eq:mean_KF}
\\[.1cm]
\frac{\ud \Sigma_t}{\ud t}
&= A \Sigma_t + \Sigma_t A^T + \sigma \sigma^T - \Sigma_t H^T
%R^{-1}
H \Sigma_t
\label{eq:var_KF}
\end{align}
\end{subequations}
where 
\begin{equation}
{\sf K}_t:=\Sigma_t H_t^{\rm T}
\label{eq:Kalman_gain}
\end{equation}
 is referred to as the Kalman gain
and the filter is initialized with the initial conditions
$\hat{X}_0$ and $\Sigma_0$ of the prior density. Table~\ref{tab:symbols-KF}
includes a list of symbols used for the Kalman filter. 

The evolution equation for the mean is a sde because of the presence
of stochastic forcing term $Z_t$ on the right-hand side.  The
evolution equation for the variance $\Sigma_t$ is an ode that does not
depend upon the observation process.  

The Kalman filter is one of the most widely used algorithm in engineering.  Although the filter describes the posterior {\em only}
in linear Gaussian settings, it is often used as an approximate
algorithm even in more general settings, e.g., by defining the
matrices $A$ and $H$ according to the Jacobians of the
mappings $a$ and $h$:
\[
A:= \frac{\partial a}{\partial x} (\hat{X}_t), \quad H:= \frac{\partial h}{\partial x} (\hat{X}_t) 
\] 
The resulting algorithm is referred to as the extended Kalman
filter:
\begin{subequations}
\begin{align}
 \ud \hat{X}_t  &= a(\hat{X}_t) \ud t + {\sf K}_t 
\Bigl(\ud
Z_t-  h(\hat{X}_t) \ud t \Bigr)
\label{eq:mean-EKF}
\\[.1cm]
\frac{\ud \Sigma_t}{\ud t}
&= A \Sigma_t + \Sigma_t A^T + \sigma (\hat{X}_t) \sigma^T(\hat{X}_t) - \Sigma_t H^T
%R^{-1}
H \Sigma_t
\label{eq:var-EKF}
\end{align}
\end{subequations}
where ${\sf K}_t=\Sigma_t H^{\rm T}$ is used as the formula for the gain.

The Kalman filter and its extensions are recursive algorithms that
process measurements in a sequential (online) fashion.  At each time
$t$, the filter computes an error $\ud Z_t-  H \hat{X}_t \ud t$
(called the {\em innovation error}) which reflects the
new information contained in the most recent measurement.  The filter
state $\hat{X}_t$ is corrected at each time step via a
$(\text{gain} \times \text{error})$ update formula.  

The error correction feedback structure (see
Fig.~\ref{fig:fig_FPF_KF}) is important on account of robustness. A filter is based on an idealized model of
an underlying stochastic dynamic process.  The self-correcting
property of the feedback provides robustness, allowing one to tolerate
a degree of uncertainty inherent in any model.

The simple intuitive nature of the update formula is invaluable in
design, testing and operation of the filter.  For example, the Kalman
gain is proportional to $H$ which scales with the signal-to-noise
ratio of the measurement model.  In practice, the gain may be `tuned'
to optimize the filter performance.  To minimize online
computations, an offline solution of the algebraic Ricatti equation
(obtained after equating the right-hand side of the variance
ode~\eqref{eq:var_KF} to zero) may be used to obtain a constant value for the
gain.          

The basic Kalman filter has also been extended to handle filtering
problems involving additional uncertainties in the signal model and
the observation model.  The resulting (approximate) algorithms are
referred to as the interacting multiple model (IMM) filter~\cite{Blom_cdc12} and the
probabilistic data association (PDA) filter~\cite{Bar-Shalom_IEEE_CSM}, respectively.  In the
PDA filter, the gain varies based on an estimate of the instantaneous
uncertainty in the measurements.  In the IMM filter, multiple Kalman
filters are run in parallel and their outputs combined to obtain an
estimate.

One explanation of the feedback control structure of the Kalman filter
is based on duality between estimation and control~\cite{kalman1960}.
Although limited to linear Gaussian problems, these considerations
also help explain the differential Ricatti equation structure for the
variance ode~\eqref{eq:var_KF}.

\medskip

Although widely used, the extended Kalman filter can suffer
from stability issues because of the very crude
approximation of the nonlinear model.  The observed divergence arises
on account of two inter-related reasons: (i) Even with Gaussian
process and measurement noise, the nonlinearity of the mappings
$a(\cdot)$, $\sigma(\cdot)$ and $h(\cdot)$ can lead to non-Gaussian
forms of the posterior density $p^*$; and (ii) the Jacobians $A$ and
$H$ used in propagating the covariance can lead to large errors in
approximation of the gain particularly if the Hessian of these
mappings is large.  These issues have necessitated development of
particle based algorithms described in the following sections.

\begin{table}[t]
\centering
\begin{tabular}{|c|c|c|}
\hline
Variable & Notation & Model\\ \hline
\multirow{2}{*}{Particle state} & \multirow{2}{*}{$X^i_t$} &
Stoch. EnKBF Eq.~\eqref{Stochastic EnKBF} \\ 
& & Deter. EnKBF Eq.~\eqref{detEnKBF} \\ \hline
Empirical variance  & $\Sigma_t^{(N)}$ & Eq.~\eqref{empcov} \\ \hline
Particle process noise & $B^i_t$ & Wiener process \\\hline
Particle meas. noise  & $W^i_t$ & Wiener process 
\\\hline
\end{tabular}
\caption{Symbols for the ensemble Kalman-Bucy filter}  
\label{tab:symbols-EnKF}
\end{table}

\section{Ensemble Kalman-Bucy Filter}
\label{sec:EnKF}

For pedagogical reasons, the ensemble Kalman-Bucy filter (EnKBF) is best
described for the linear Gaussian problem -- also the approach taken
in this section.  The extension to the nonlinear non-Gaussian problem
is then immediate, similar to the extension from the Kalman filter to
the extended Kalman filter.

Even in linear Gaussian settings, a particle filter may be a
computationally efficient option for problems with very large state
dimension $d$ (e.g., weather models in meteorology).  For large $d$,
the computational bottleneck in simulating a Kalman filter arises due
to propagation of the covariance matrix according to the differential 
Riccati equation~\eqref{eq:var_KF}.  This computation scales as $O(d^2)$ in
memory.  In an EnKBF implementation, one replaces the exact propagation
of the covariance matrix by an empirical approximation with $N$ particles
\begin{align}
  \qquad \Sigma_t^{(N)} &= \frac{1}{N-1} \sum_{i=1}^N (X_t^i - \hat X_t^{(N)})
(X_t^i- \hat X_t^{(N)})^{\rm T} \label{empcov}
\end{align}
This computation scales as $O(Nd)$. The same reduction in computational cost can be achieved 
by a reduced rank Kalman filter. However, the connection to empirical measures (\ref{eqn:empirical})
is crucial to the application of the EnKBF to nonlinear dynamical systems.

The EnKF algorithm was first developed in a discrete-time
setting~\cite{jdw:BurgersLeeuwenEvensen1998}. Since then various formulations of the
EnKF have been proposed
\cite{jdw:LawStuartZygalakis2015,jdw:ReichCotter2015}. Below we state
two continuous-time formulations of the EnKBF.  Table~\ref{tab:symbols-EnKF}
includes a list of symbols used for these formulations.

\subsection{Stochastic EnKBF}

The conceptual idea of the stochastic EnKBF algorithm is to introduce a
zero mean perturbation (noise term) in the innovation error to achieve
consistency for the variance update.  In the continuous-time stochastic EnKBF
algorithm, the particles evolve according to
\begin{equation} \label{Stochastic EnKBF}
\ud X_t^i = A X_t^i\ud t + \sigma \ud B_t^i + \Sigma^{(N)}_tH^{{\rm T}}
\Big( \ud {Z}_t^i -HX^i_t\ud t + \ud W_t^i\Big)
\end{equation}
for $i=1,\ldots,N$, where $X_t^i\in\Re^d$ is the state of the
$i^\text{th}$ particle at time $t$, the initial condition $X^i_0\sim
p_0^*$, $B^i_t$ is a standard Wiener
process, and $W_t^i$ is a standard Wiener process assumed to be independent
of $X_0^i,\,B_t^i,\,X_t,\,Z_t$ ~\cite{jdw:LawStuartZygalakis2015}.  The variance $\Sigma^{(N)}_t$ is obtained
empirically using~\eqref{empcov}.  Note that the $N$ particles only interact through the common covariance matrix $\Sigma_t^{(N)}$. 

The idea of introducing a noise process first appeared for the
discrete-time EnKF.  The derivation of the continuous-time stochastic
EnKBF can be found in~\cite{jdw:LawStuartZygalakis2015}
or~\cite{jdw:Reich2011}.  It is based on a limiting argument whereby
the discrete-time update step is formally viewed as an Euler-Maruyama
discretization of a stochastic SDE.  For the linear Gaussian problem,
the stochastic EnKBF algorithm is consistent in the limit as
$N\rightarrow\infty$.  This means that the conditional distribution of
$X_t^i$ is Gaussian whose mean and variance evolve according the
Kalman filter, equations~\eqref{eq:mean_KF} and~\eqref{eq:var_KF},
respectively. 

The update formula used in~\eqref{Stochastic EnKBF} is not unique.
A deterministic analogue is described next.

\subsection{Deterministic EnKBF}
A deterministic variant of the EnKBF (first proposed
in \cite{jdw:BergemannReich2012}) is given by:
\begin{equation} \label{detEnKBF}
\ud X_t^i = A X_t^i \ud t + \sigma \ud B_t^i 
+ \Sigma_t^{(N)}H^{\rm T} \left( \ud Z_t - \frac{HX_t^i  + H\hat X_t^{(N)}}{2}
{\rm d}t \right)
\end{equation}
for $i=1,\ldots,N$.
A proof of the consistency of this deterministic variant of the EnKBF
for linear systems can be found in
\cite{jdw:deWiljesReichStannat2016}. There are close parallels between
the deterministic EnKBF and the FPF which is explored further in Section
\ref{sec:FPF}.  In the deterministic formulation of the EnKBF the interaction between the $N$ particles arises through 
the covariance matrix $\Sigma^{(N)}_t$ and the mean $\hat X_t^{(N)}$.

Although the stochastic and the deterministic EnKBF algorithms are
consistent for the linear Gaussian problem, they can be easily
extended to the nonlinear non-Gaussian settings.  However, the resulting
algorithm will in general not be consistent.

\subsection{Well-posedness and accuracy of the EnKBF}
Recent research on EnKF has focussed on the long term behavior and
accuracy of filters applicable for nonlinear data assimilation
\cite{sr:DelMoral16,jdw:deWiljesReichStannat2016,jdw:KellyLawStuart2014,jdw:KellyStuart2016}. In
particular the mathematical justification for the feasibility of the
EnKF and its continuous counterpart in the small ensemble limit are of
interest. These studies of the accuracy for a finite ensemble are of
exceptional importance due to the fact that a large number of ensemble
members is not an option from a computational point of view in many
applicational areas. The authors of \cite{jdw:KellyLawStuart2014}
show mean-squared asymptotic accuracy in the large-time
limit where a particular type of variance inflation is deployed for the
stochastic EnKBF. Well-posedness of the discrete and
continuous formulation of the EnKF is also shown. Similar results
concerning the well-posedness and accuracy for the deterministic
variant~\eqref{detEnKBF} of the EnKBF are derived in
\cite{jdw:deWiljesReichStannat2016}.
A fully observed system is assumed in deriving these accuracy and
well-posedness results.  An investigation of the
well-posedness for partially observed systems is particularly relevant
as the update step in such cases can cause a divergence of the
estimate in the sense that the signal is lost or the values of the
estimate reach machine infinity \cite{jdw:KellyMajdaTong2015}.

\section{Feedback Particle Filter}
\label{sec:FPF}

The FPF is a controlled sde:
\begin{equation}
\begin{aligned}
\ud X^i_t = &a(X^i_t) \ud t + \sigma(X_t^i) \ud B^i_t \\&+ \underbrace{\k_t(X^i_t) \circ (\ud Z_t -
\frac{h(X^i_t) + \hat{h}_t}{2}\ud t)}_{\text{update}},
\quad X_0^i \sim p_0^*
\end{aligned}
\label{eqn:particle_filter_nonlin_intro}
\end{equation}
where (similar to EnKBF) $X_t^i\in\Re^d$ is the state of the
$i^\text{th}$ particle at time $t$, the initial condition $X^i_0\sim
p_0^*$, $B^i_t$ is a standard Wiener process, and $\hat{h}_t :=
\E[h(X_t^i)|\mathcal{Z}_t]$.  Both  $B^i_t$ and $X^i_0$ are mutually
independent and also independent of $X_t,Z_t$.  The $\circ$ in the
update term indicates
that the sde is expressed in its Stratonovich form.

The gain function $\k_t$ is vector-valued (with dimension
$d\times m$) and it needs to be obtained for each fixed time $t$.  The
gain function is defined as a solution of a pde introduced in the
following subsection.  
For the linear Gaussian problem, $\k_t$ is the Kalman gain.   For the
general nonlinear non-Gaussian, the gain function needs to be
numerically approximated.  Algorithms for this are also summarized in
the following subsection.

\begin{remark}\label{rem:remStrato}
Given that the Stratonovich form provides a mathematical
interpretation of the (formal) ode model \cite[see Section 3.3 of the sde
textbook by {\O}ksendal]{oksendal2013}, we also obtain the (formal) ode
model of the filter. Denoting $Y_t \doteq \frac{\ud Z_t}{\ud t} $ and white noise
process $\dot{B}^i_t \doteq \frac{\ud B_t^i}{\ud t} $, the ODE model of the
filter is given by,
\begin{equation}
\frac{\ud X^i_t}{\ud t} = a(X^i_t) + \sigma(X_t^i) \dot{B}^i_t +  \k(X^i,t) \left( Y_t - \frac{1}{2}    (h(X^i_t) + \hat{h})\right),  \nonumber%\label{eqn:FPF_ODE}
\end{equation}
for $i=1,\ldots,N$. The feedback particle filter thus provides a generalization of the
Kalman filter to nonlinear systems, where the innovation error-based
feedback structure of the control is preserved (see
Fig.~\ref{fig:fig_FPF_KF}).  For the linear Gaussian case, the gain function
is the Kalman gain.  For the nonlinear case, the Kalman gain is
replaced by a nonlinear function of the state (See Fig.~\ref{fig:const-gain-approx}).
\end{remark}

\begin{figure*}
    \centering
 \includegraphics[scale=0.31]{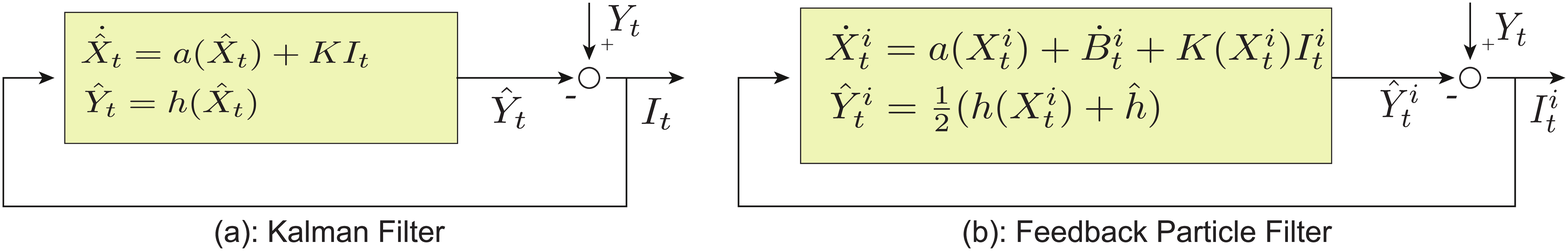}
   \vspace{-0.15in}
\caption{Innovation error-based feedback structure for the (a) Kalman
  filter and (b) nonlinear feedback particle filter. 
}\vspace{-.1cm}
    \label{fig:fig_FPF_KF}
\end{figure*}

\begin{remark}
It is shown in Appendix \ref{sec:theory} that, under the condition that the
gain function can be computed exactly, FPF is an exact algorithm.  That
is, if the initial condition $X^i_0$ is sampled from the prior $p_0^*$
then 
\[
\Prob [X_t \in A\mid \clZ_t ] = \Prob [X_t^i \in A\mid \clZ_t ], \quad \forall\;A\subset \mathbb{R}^d,\;\;t>0.
\]
In a numerical implementation, a finite number, $N$, of particles is simulated and 
$\Prob [X_t^i \in A\mid \clZ_t ] \approx \frac{1}{N}\sum_{i=1}^N \ind
[ X^i_t\in A]$ by the Law of Large Numbers (LLN). 

The considerations in the Appendix are described in a more general
setting, e.g., applicable to stochastic processes $X_t$ and $X_t^i$
evolving on manifolds.  This also explains why 
the update formula has a Stratonovich form.  For sdes on a manifold,
it is well known that the Stratonovich form is invariant to coordinate
transformations (i.e., intrinsic) while the Ito~form is not.  
A more in-depth discussion of the FPF for Lie groups appears
in~\cite{Chi_ACC16}.           
\end{remark}

Table~\ref{tab:symbols-FPF}
includes a list of symbols used for the FPF.

\begin{table}[t]
\centering
\begin{tabular}{|c|c|c|}
\hline
Variable & Notation & Model\\ \hline
Particle state & $X^i_t$ & FPF Eq.~\eqref{eqn:particle_filter_nonlin_intro} \\ \hline
Gain function  & $\K_t(x)=\nabla \phi(x)$ & Poisson Eq.~\eqref{eqn:EL_phi_intro} \\ \hline 
\multirow{4}{*}{Particle gain} & \multirow{4}{*}{ $\K^i=\K(X^i_t)$} &  Constant gain~\eqref{eq:const-gain-approx} \\ & & Galerkin~\eqref{eq:galerkin-Ki} \\ & & Kernel-based~\eqref{eq:kernelk}\\ & & Optimal coupling~\eqref{eqn:OT}   
\\\hline
\end{tabular}
\caption{Symbols for feedback particle filter}  
\label{tab:symbols-FPF}
\end{table}

\subsection{Gain function}

For pedagogical reasons primarily to do with notational convenience, the gain function
is defined here for the case of scalar-valued observation\footnote{The extension to
  multi-valued observation is straightforward and appears in~\cite{yang2016}.}.  In this case, the gain function $\k_t$ is defined in terms of the solution of the
weighted Poisson equation: 
\begin{equation}
\label{eqn:EL_phi_intro}
\begin{aligned}
-\nabla \cdot (\rho(x) \nabla \phi(x)) & = (h (x)-\hat{h}) \rho(x),\quad x\in \Re^d
\\
\int \phi (x) \rho(x) \ud x & = 0
\end{aligned}
\end{equation}
where  $\hat{h} :=
\int h(x)\rho(x)\ud x$, $\nabla$ and $\nabla \cdot $ denote the
gradient and the divergence operators, respectively, and at time $t$, 
$\rho(x)=p(x,t)$ denotes the density of $X_t^i$\footnote{Although this
  paper is limited to $\Re^d$, the proposed algorithm is applicable to
 nonlinear filtering problems on differential manifolds, e.g., matrix Lie
  groups (For an intrinsic form of the Poisson equation, see~\cite{Chi_ACC16}).  
  For domains with boundary, the pde is accompanied by a Neumann boundary condition:
\[
\nabla \phi(x) \cdot {n}(x) = 0
\] 
for all $x$ on the boundary of the domain where ${n}(x)$ is a unit
normal vector at the boundary point $x$.}. 
In terms of the solution $\phi(x)$ of~\eqref{eqn:EL_phi_intro}, %with $\rho(x)=p(x,t)$, 
the gain function at time $t$ is given by
\begin{align}
\k_t(x)= \nabla \phi(x)\,.
\label{eqn:gradient_gain_fn_intro}
\end{align}

\begin{remark}
The gain function $\k_t(x)$ is not uniquely defined through the
filtering problem.  Formula~\eqref{eqn:gradient_gain_fn_intro}
represents one choice of the gain function.  More generally, it is
sufficient to require that $\k_t=\k$ is a solution of
\begin{equation}
-\nabla \cdot (\rho(x) \k(x))  = (h (x)-\hat{h}) \rho(x),\quad x\in
\Re^d
\label{eqn:div_eqn}
\end{equation}
with $\rho(x)=p(x,t)$ at time
$t$.  

One justification for choosing
the gradient form solution, as in~\eqref{eqn:gradient_gain_fn_intro},
is its $L^2$ optimality. 
The general
solution of~\eqref{eqn:div_eqn} is given by
\[
\k = \nabla \phi + v
\]
where $\phi$ is the solution of~\eqref{eqn:EL_phi_intro} and $v$
solves $\nabla \cdot(\pr v)=0$.  It is easy to see that
\[
{\sf E}[|\k|^2]  = {\sf E}[|\nabla \phi|^2] + {\sf E}[|v|^2].
\]
Therefore, $\k=\nabla\phi$ is the minimum $L^2$-norm solution
of~\eqref{eqn:div_eqn}.  By interpreting the $L^2$ norm as the kinetic
energy, the gain function $\k_t = \nabla \phi$, defined through
(\ref{eqn:EL_phi_intro}), is seen to be optimal in the sense of
optimal transportation \cite{villani2003,evans}.

An alternative solution of~\eqref{eqn:div_eqn} is provided through the 
definition
\[
\k_t(x) = \frac{1}{\rho(x)}\nabla \tilde \phi(x)
\]
which leads to a standard Poisson equation in the unknown potential
$\tilde \phi$ for which the fundamental solution is explicitly
known. This fact is exploited in the interacting particle filter representations of \cite{crisan2007,crisan10}.
\label{remark:optimal-K}
\end{remark}

\medskip

There are two special cases of (\ref{eqn:EL_phi_intro}) -- summarized as part of the following two
examples -- where the exact solution can be found.

\begin{example} In the scalar case (where $d=1$), the Poisson equation is:
\begin{equation*}
-\frac{1}{\rho(x)}\frac{\ud }{\ud x}(\rho(x) \frac{\ud \phi}{\ud
  x}(x)) = h-\hah.
\end{equation*}
Integrating once yields the solution explicitly,
\begin{equation}
\begin{aligned}
\k(x) = \frac{\ud \phi}{\ud x}(x) &= -\frac{1}{\rho(x)}\int_{-\infty}^x
\rho(z)(h(z)-\hah)\ud z.
\end{aligned}
\label{eq:scalar}
\end{equation}

For the particular choice of $\rho$ as the sum of two Gaussians
${\cal N}(-1,\sigma^2)$ and ${\cal N}(+1,\sigma^2)$ with $\sigma^2=0.2$ and $h(x)=x$, the solution
obtained using~\eqref{eq:scalar} is depicted in \Fig{fig:truesol}.

\begin{figure}
\centering
\includegraphics[width=0.8\columnwidth]{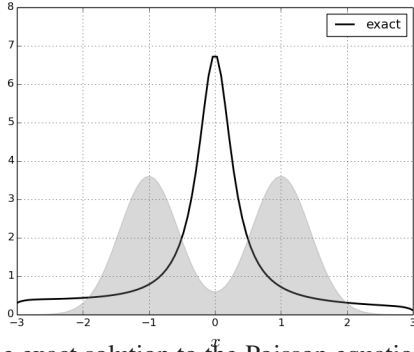}
\vspace*{-0.2in}
\caption{The exact solution to the Poisson equation using the
  formula~\eqref{eq:scalar}.  The density $\rho$ is the sum of two Gaussians
$N(-1,\sigma^2)$ and $N(+1,\sigma^2)$, and $h(x)=x$.  The density is depicted as the shaded curve in the background.}
\vspace*{-0.1in}
\label{fig:truesol}
\end{figure}
\label{ex:scalar}
\end{example}

\begin{example} Suppose the density $\rho$ is a
  Gaussian ${\cal N}(\mu,\Sigma)$.  The observation function $h(x) =
  Hx$, where $H \in \Re^{1\times d}$.  Then, $\phi = x^{\rm T}\Sigma H^{\rm T}
  $ and the gain function $\k = \Sigma \, H^{\rm T}$ is the Kalman
  gain.
\end{example}

\medskip

In the general non-Gaussian case, the solution is not known in an
explicit form and must be numerically approximated.  Note that even in
the two exact cases, one would need to numerically approximate the solution
because the density $\rho$ is not available in an explicit form.

The problem statement for numerical approximation is as follows:

\medskip

\noindent \textbf{Problem statement:} Given $N$ 
samples $\{X^1,\hdots,X^i,\hdots,X^N\}$ drawn i.i.d. from $\rho$, approximate
the vector-valued gain function$\{\K^1,\hdots,\K^i,\hdots,\K^N\}$, where
$\K^i:=\K(X^i)=\nabla\phi(X^i)$.  
The density $\rho$ is not explicitly known.  

\medskip

Four numerical
algorithms for approximation of the gain function appear in the
following four subsections\footnote{These algorithms are based on the
  existence-uniqueness theory for solution $\phi$ of the Poisson
  equation pde~\eqref{eqn:EL_phi_intro}, as described in~\cite{laugesen15}.}.  

\subsection{Constant Gain Approximation}

The constant gain approximation is the best -- in the least-square sense -- constant approximation of the
gain function (see Figure~\ref{fig:const-gain-approx}).
Mathematically, it is obtained by considering the following least-square
optimization problem:
\begin{equation*}
\kappa^\ast = \arg \min_{\kappa \in \Re^d} {\sf E} \, [|\k - \kappa|^2]
\end{equation*}
By using a standard sum of the squares argument, 
$
\kappa^\ast = {\sf E}[\k].
$
By multiplying both sides of the pde~\eqref{eqn:EL_phi_intro} by $x$
and integrating by parts, the expected value is computed
explicitly as
\begin{align*}
\kappa^* = \Expect[\k] &=\int_{\Re^d} (h(x)-\hah)\;
x \; \rho(x) \ud x
\end{align*}
The integral is evaluated empirically to obtain the following
approximate formula for the gain:
\begin{equation}
\k^i \equiv \frac{1}{N}\sum_{j=1}^N\; (h(X^j)-\hat{h}^{(N)}) \; X^j
\label{eq:const-gain-approx}
\end{equation}
where $\hat{h}^{(N)} = N^{-1} \sum_{j=1}^N h(X^j)$. 
The formula~\eqref{eq:const-gain-approx} is referred to as the {\em constant gain approximation}
of the gain function;
cf.,~\cite{yang2016}.  It is a popular choice in
applications~\cite{yang2016,stano2014,tilton2013,berntorp2015} and is equivalent to the approximation used in the deterministic and stochastic EnKBF~\cite{jdw:ReichCotter2015,jdw:LawStuartZygalakis2015,jdw:Reich2011,jdw:BergemannReich2012,jdw:deWiljesReichStannat2016}.

\begin{example}
Consider the linear case where $h(x) = Hx$.  The constant gain
approximation formula equals
\[
\k^i = \frac{1}{N}\sum_{j=1}^N\; (H X^j -H \hat{X}^{(N)}) \; X^i =
\Sigma^{(N)} H, 
\]
where $\hat{X}^{(N)}:=\frac{1}{N}\sum_{i=1}^NX^i$ and $\Sigma^N=\frac{1}{N}\sum_{i=1}^N(X^i-\hat{X}^{(N)})(X^i-\hat{X}^{(N)})^{\rm T}$ 
are the empirical mean and
the empirical variance, respectively.
That is, for the linear Gaussian case, the FPF algorithm with the
constant gain approximation gives the deterministic EnKF algorithm. 
\end{example}

\medskip

\begin{figure}[t]
\centering
\includegraphics[width=0.8\columnwidth]{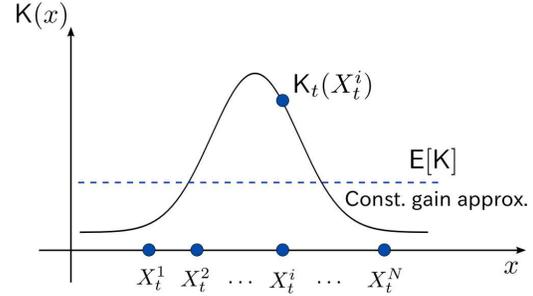}
\caption{Constant gain approximation in the feedback particle filter}
\label{fig:const-gain-approx}
\end{figure}

\subsection{Galerkin Approximation}
\label{sec:Galerkin}

The Galerkin approximation is a generalization of the constant gain
approximation where the gain function $\k=\nabla\phi$ is now
approximated in a finite-dimensional subspace
$S:=\text{span}\{\psi_1,\ldots,\psi_M\}$\footnote{$S$ is a
  finite-dimensional subspace in the Sobolev space
  $H^1_0(\Re^d;\rho)$ -- defined as the space of functions $f$ that
  are square-integrable with respect to density $\rho$ and whose
  (weak) derivatives are also square-integrable with respect to
  density $\rho$.  $H_0^1$ is the appropriate space for the solution
  $\phi$ of the Poisson equation~\eqref{eqn:EL_phi_intro}.}.
Mathematically, the Galerkin solution $\nabla\phiM$ is defined as the optimal
least-square approximation of $\nabla\phi$ in $S$, i.e, 
\begin{equation*}
\phiM =\argmin_{\psi \in S} ~{\sf E}[|\nabla \phi-\nabla \psi|^2]
\end{equation*}
The least-square solution is easily obtained by applying the projection theorem which gives
\begin{equation}
{\sf E}[\nabla \phiM \cdot \nabla \psi] = {\sf E}[(h-\hat{h})\,\psi],\quad
\forall \; \psi \in S
\label{eq:Poisson-weak-finite}
\end{equation}

By denoting $(\psi_1(x),\psi_2(x),\hdots,\psi_M(x))=:\psi(x)$ and
expressing $\phiM(x) = c \cdot \psi(x)$, the finite dimensional
system~\eqref{eq:Poisson-weak-finite} is expressed as 
a linear matrix equation
\begin{equation}
Ac=b
\label{eq:Acb}
\end{equation}
where $A$ is a $M\times M$ matrix and $b$ is a $M\times 1$ vector
whose entries are given by the respective formulae:
\begin{align*}
[A]_{lk} & = {\sf E}[\nabla\psi_l \cdot \nabla
\psi_k]\\
[b]_l & =  {\sf E}[(h-\hat{h})\,\psi_l]
\end{align*}

\begin{example} Two types of approximations follow from consideration
  of two types of basis functions:
\begin{romannum}
\item The constant gain approximation is obtained by taking basis
  functions as $\psi_l(x)=x_l$ for $l=1,\ldots,d$.  With this choice,
  $A$ is the identity matrix and the Galerkin gain function is a constant vector:
\begin{equation*}
\k (x)= \int x\, (h(x)-\hat{h})\, \pr(x)\ud x
\end{equation*}
It's empirical approximation is
\[
\k^i \equiv \frac{1}{N}\sum_{j=1}^N\; (h(X^j)-\hat{h}^{(N)}) \; X^j
\]

\item With a single basis function $\psi(x) = h(x)$, the Galerkin solution is
\[
\k(x) = \frac{\int (h(x) - \hat{h})^2 \rho(x) \ud x}{\int |\nabla
  h(x) |^2 \rho(x) \ud x} \; \nabla h(x)
\]
It's empirical approximation is obtained as
\[
\k^i =\frac{\sum_{j=1}^N (h(X_t^i) -
  \hat{h}^{(N)})^2}{\sum_{j=1}^N |\nabla h(X_t^i) |^2}\; \nabla h(X_t^i)
\] 
\end{romannum}
\label{ex:basis}
\end{example}
\begin{algorithm}[t]
 \caption{Constant gain approximation}
 \begin{algorithmic}[1]
     \REQUIRE $\{X^i\}_{i=1}^N$, $\{h(X^i)\}_{i=1}^N$, 
     \ENSURE $\{\K^i\}_{i=1}^N$, \medskip
      \STATE Calculate $\hat{h}^{(N)}=\frac{1}{N}\sum_{i=1}^N h(X^i)$\medskip
     \STATE Calculate $\K_c=\frac{1}{N}\sum_{i=1}^N\; (h(X^i)-\hat{h}^{(N)}) \; X^i$ \STATE $K^i=\K_c$ for $i=1,\ldots,N$  
 \end{algorithmic}
 \label{alg:const-gain}
\end{algorithm}
\medskip

\begin{algorithm}[t]
 \caption{Galerkin approximation of the gain function}
 \begin{algorithmic}[1]
     \REQUIRE $\{X^i\}_{i=1}^N$, $\{h(X^i)\}_{i=1}^N$, $\{\psi_1,\ldots,\psi_M\}$, 
     \ENSURE $\{\K^i\}_{i=1}^N$, \medskip
      \STATE$\hat{h}^{(N)}=\frac{1}{N}\sum_{i=1}^N h(X^i)$
     \STATE $[A^{(N)}]_{lk} :=\frac{1}{N}\sum_{i=1}^N \nabla \psi_l(X^i) \cdot \nabla \psi_k(X^i),$ for $l,k=1,\ldots,M$
     \STATE $[b^{(N)}]_k :=\frac{1}{N}\sum_{i=1}^N
     \psi_k(X^i)(h(X^i)-\hat{h}^{(N)})$, for $k=1,\ldots,M$ 
     \STATE Calculate $c^{(N)}$ by solving $A^{(N)} c^{(N)}= b^{(N)}$
     \STATE $\K^i=  \sum_{k=1}^M c^{(N)}_k \nabla \psi_k(X^i)$
 \end{algorithmic}
 \label{alg:galerkin}
\end{algorithm}

In practice,  the matrix $A$ and the vector $b$ are approximated
empirically, and the equation~\eqref{eq:Acb} solved to obtain the
empirical approximation of $c$, denoted as $c^{(N)}$ (see
Table~\ref{alg:galerkin} for the Galerkin algorithm).  In terms of
this empirical approximation, the gain function is approximated as
\begin{equation}
\K^i = \nabla\phi^{(M,N)}(X^i) := c^{(N)} \cdot \nabla \psi(X^i)
\label{eq:galerkin-Ki}
\end{equation}

The choice of basis function is problem dependent.  In Euclidean
settings, the linear basis functions are standard and lead to constant
gain approximation as discussed in  
Example~\ref{ex:basis}.  A straightforward extension is to choose
quadratic and higher order polynomials as basis functions. 
However, this approach does not scale well with the dimension of
problem: The number of basis functions $M$ scales at least
linearly with dimension, and the Galerkin algorithm involves inverting
a $M\times M$ matrix. 
This motivates nonparametric and data driven approaches where (a
small number of) basis functions can be selected in an adaptive fashion.  One such
algorithm, proposed in~\cite{berntorp2016}, is based on the Karhunen-Loeve
expansion.

In the next two subsections, alternate data-driven algorithms
are described.  One advantage of these algorithms is that they do not
require a selection of basis functions.

\subsection{Kernel-based Approximation}
\label{sec:kernel}

The linear operator $\frac{1}{\rho}\nabla \cdot (\rho
\nabla)=:\Delta_{\rho}$ for the pde~\eqref{eqn:EL_phi_intro} is a generator of a
Markov semigroup, denoted as $e^{\epsilon\Delta_\rho}$ for
$\epsilon>0$.  It follows that the solution $\phi$ of~\eqref{eqn:EL_phi_intro} is equivalently expressed as,
for any fixed $\epsilon>0$,
\begin{equation}
\phi = e^{\epsilon\Delta_\rho} \phi+ \int_0^\epsilon e^{s\Delta_\rho} (h-\hat{h}) \ud s.
\label{eq:fixed1_n}
\end{equation}

The fixed-point representation is useful because $e^{\epsilon
  \Delta_\rho}$ can be approximated by a finite-rank operator
\begin{equation*}
\TepsN f(x) :=\sum_{i=1}^N \kepsN(x,X^i)f(X^i),
\label{eq:TepsN}
\end{equation*}
where the kernel
\begin{equation*}
\kepsN (x,y) = \frac{1}{\nepsN(x)}\frac{\geps(x-y)}{\sqrt{\frac{1}{N}\sum_{i=1}^N \geps(x-X^i)}\sqrt{\frac{1}{N}\sum_{i=1}^N\geps(y-X^i)}}
\end{equation*}
is expressed in terms of the Gaussian kernel $\geps
(z):={(4\pi\epsilon)}^{-\frac{d}{2}}\exp{(-\frac{|z|^2}{4\epsilon})}$
for $z\in\Re^d$,
and $\nepsN(x)$ is a normalization factor chosen such that $\TepsN
1=1$. It is shown in \cite{coifman,hein2006} that
$e^{\epsilon\Delta\pr}\approx \TepsN$ as $\epsilon \downarrow 0$ and
$N \to \infty$.

The approximation of the fixed-point problem~\eqref{eq:fixed1_n} is
obtained as
\begin{align}
\phiepsN = \TepsN \phiepsN + \epsilon(h-\hat{h}),
\label{eq:fixed-point-epsN}
\end{align}
where $\int_0^\epsilon e^{s\Delta_\pr}(h-\hat{h})\ud s \approx
\epsilon(h-\hat{h})$ for small $\epsilon>0$.   The method of successive approximation is
used to solve the fixed-point equation for
$\phiepsN$.  In a recursive simulation, the algorithm is initialized with
the solution from the previous time-step.

The gain function is obtained by taking the gradient of
the two sides of~\eqref{eq:fixed-point-epsN}.  For this purpose, it is
useful to first define a finite-rank operator: 
\begin{equation}
\begin{aligned}
\nabla &\TepsN f(x)  := \sum_{i=1}^N\nabla \kepsN(x,X^i)f(X^i) \\
 & = \frac{1}{2\epsilon}\left[\sum_{i=1}^N \kepsN(x,X^i)f(X^i)\left(X^i - \sum_{j=1}^N\kepsN(x,X^j)X^j\right) \right]
 \label{eq:gradTeps}
\end{aligned}
\end{equation}
In terms of this operator, the gain function is approximated as
\begin{equation}
\K^i = \nabla \TepsN \phiepsN(X^i) +  \epsilon\nabla \TepsN(h-\hat{h}^{(N)})(X^i)
\label{eq:uepsN}
\end{equation}
where $\phiepsN$ on the righthand-side is the solution
of~\eqref{eq:fixed-point-epsN}. 

For $i,j\in\{1,2,\hdots,N\}$, denote
\[
a_{ij} := \frac{1}{2\epsilon}\kepsN(X^i,X^j)\left(r_j -\sum_{l=1}^N \kepsN(X^i,X^l)r_l\right)
\]
where $r_i:=\phiepsN(X^i) + \epsilon h(X^i) - \epsilon\hat{h}^{(N)}$.
Then, the formula~\eqref{eq:uepsN} is succinctly expressed as 
\begin{equation}
\label{eq:kernelk}
\K^i = \sum_{j=1}^N a_{ij}X^j
\end{equation}   
%where $S_{ij} = \frac{1}{2\epsilon}\kepsN(X^i,X^j)\left(r_j -\sum_{l=1}^N \kepsN(X^i,X^l)r_l\right)$ with $r_i=\phiepsN(X^i) + \epsilon h(X^i) - \epsilon\hat{h}^{(N)}$. 

It is easy to verify that $\sum_{j=1}^N a_{ij} = 0$ and as
$\epsilon\rightarrow\infty$, $a_{ij} = {N}^{-1}(h(X^j) - \hat{h}^{(N)})$.
Therefore, as 
$\epsilon\rightarrow\infty$, $\K^i$ equals the constant gain
approximation formula~\eqref{eq:const-gain-approx}.     

\begin{algorithm}[h]
 \caption{Kernel-based approximation of the gain function}
 \begin{algorithmic}[1]
     \REQUIRE $\{X^i\}_{i=1}^N$, $\{h(X^i)\}_{i=1}^N$, $\Phi_{\text{prev}}$, $L$
     \ENSURE $\{\K^i\}_{i=1}^N$ \medskip
     \STATE Calculate $g_{ij}:=\exp(-|X^i-X^j|^2/4\epsilon)$ for $i,j=1$ to $N$\medskip
     \STATE Calculate $k_{ij}:=\frac{g_{ij}}{\sqrt{\sum_l g_{il}}\sqrt{\sum_l g_{jl}}}$ for $i,j=1$ to $N$
     \STATE Calculate $T_{ij}:=\frac{k_{ij}}{\sum_l k_{il}}$ for $i,j=1$ to $N$
     \STATE Calculate $\hat{h}^{(N)}=\frac{1}{N}\sum_{i=1}^N h(X^i)$\medskip
     \STATE Initialize $\Phi_i=\Phi_{\text{prev},i}$ for $i=1$ to $N$
     \medskip
     \FOR {$l=1$ to  $L$}
     \STATE Calculate $\Phi_i= \sum_{j=1}^N T_{ij} \Phi_j + \epsilon (h(X^i)-\hat{h}^{(N)})$\medskip
     \STATE Calculate $\Phi_i = \Phi_i - \frac{1}{N}\sum_{j=1}^N \Phi_j$
     \ENDFOR
     \STATE Calculate $r_i = \Phi_i + \epsilon(h(X^i)-\hat{h}^{(N)})$ 
     \STATE Calculate $a_{ij} = \frac{1}{2\epsilon}T_{ij} \left(r_j -\sum_{l=1}^N T_{il}r_l\right)$
     \STATE Calculate $\K^i = \sum_{j=1}^N a_{ij} X^j$
 \end{algorithmic}
 \label{alg:kernel}
\end{algorithm}

%%%%%%%%%%%%%%%%%%%%%%%%%%

\subsection{Optimal Coupling-based Approximation}
\label{sec:optimal}
Optimal coupling-based approximation is another non-parametric
approach to directly approximate the gain function $\k^i$ from the ensemble
$\{X^i\}_{i=1}^N$.  The algorithm is presented here for the first time
in the context of FPF.

This approximation is based upon a continuous-time reformulation of
the recently developed ensemble transform for optimally transporting
(coupling) measures~\cite{jdw:ReichCotter2015}.  The relationship to
the gain function approximation is as follows: 
Define an $\epsilon$-parametrized family of densities by
$
\rho_{\epsilon}(x) := \rho(x) (1 + \epsilon (h(x) - \hat{h}(x)))
$ for $\epsilon >0$ sufficiently small and consider the optimal transport problem
\begin{equation}\label{eq:wass}
\begin{aligned}
\text{Objective:} & \quad \quad \min_{S_\epsilon} \; {\sf E}\,[|S_\epsilon(X)-X|^2]\\
\text{Constraints:} &\quad \quad X \sim \rho,\quad S_{\epsilon}(X)\sim \rho_{\epsilon} 
\end{aligned}
\end{equation}
The solution to this problem, denoted as $S_\epsilon$, is referred to
as the optimal transport map. It is shown in Appendix~\ref{sec:lp_justification} that
$\frac{\ud S_\epsilon}{\ud \epsilon} \big|_{\epsilon=0}= \K$.

The ensemble transform is a non-parametric algorithm to approximate
the solution $S_\epsilon$ of the optimal transportation problem given {\em only} $N$
samples $X^i$ drawn from $\rho$.  For the problem of gain function approximation, the
algorithm involves first solving the linear program:
\begin{equation}\label{eq:lp}
\begin{aligned}  
\text{Objective:} & \quad \quad \min_{\{t_{ij}\} } \quad
\sum_{i=1}^N \sum_{j=1}^N \; t_{ij} \,|X^i-X^j|^2 \\
\text{Constraints:} & \quad \quad \sum_{j=1}^N t_{ij} = \frac{1}{N}, \quad
\sum_{i=1}^N t_{ij} = \frac{1+ \epsilon (h(X^j) - \hat h^{(N)})}{N},\\
&\quad \quad \quad t_{ij} \ge 0.
\end{aligned}
\end{equation}
The solution, denoted as $t_{ij}^\ast$, is referred to as the optimal
coupling, where the coupling
constants $t_{ij}^\ast$ have the interpretation of the joint probabilities. 
The two equality constraints arise due to the specification of the two
marginals $\rho$ and $\rho_{\epsilon}$ in~\eqref{eq:wass} where it is
noted that the particles $X^i$ are sampled i.i.d. 
from $\rho$. The optimal value is an approximation
of the optimal value of the objective in~\eqref{eq:wass}.  The latter
is the celebrated Wasserstein distance between $\rho$ and
$\rho_{\epsilon}$. 

In terms of the optimal solution of the
linear program~\eqref{eq:lp}, an approximation to the gain function at $X^i$ 
is obtained as 
\begin{equation}\label{eqn:OT}
\K^i :=  \sum_{j=1}^N a_{ij} \,X^j, \qquad a_{ij} =  \frac{t_{ij}^\ast -\delta_{ij}}{\epsilon} \,,
\end{equation}
where $\delta_{ij}$ is the Dirac delta tensor ($\delta_{ij}=1$ if
$i=j$ and $0$ otherwise).   In practice, a finite $\epsilon>0$ is
appropriately chosen.  The approximation becomes exact as   
$\epsilon \downarrow 0$ and $N\to \infty$. 

The approximation~\eqref{eqn:OT} is structurally similar to the
constant gain approximation formula~\eqref{eq:const-gain-approx} and
also the kernel-gain approximation formula~\eqref{eq:kernelk}.  In
all three cases, the gain $\k^i$ at the $i^{\text{th}}$ particle is
approximated as a linear combination of the particle states
$\{X^j\}_{j=1}^N$.  Such approximations are computationally attractive whenever
$N \ll d$, i.e., when the dimension of state space is high but the dynamics is confined to a low-dimensional
subset which, however, is not a priori known.

\begin{algorithm}[t]
 \caption{Optimal Coupling approximation of the gain function}
 \begin{algorithmic}[1]
     \REQUIRE $\{X^i\}_{i=1}^N$, $\{h(X^i)\}_{i=1}^N$, $\epsilon$
     \ENSURE $\{\K^i\}_{i=1}^N$ \medskip
     \STATE Calculate $d_{ij}:=|X^i-X^j|^2$ for $i,j=1$ to $N$\medskip
     \STATE Calculate $\hat{h}^{(N)} = \frac{1}{N}\sum_{i=1}^N h(X^i)$
%     \STATE Calculate $\rho_{i}:=\frac{1}{N}$ and ${\rho_\epsilon}_i=\frac{1}{N}(1+\epsilon(h(X^i)-\hat{h}^{(N)}))$ for $i=1$ to $N$
     \STATE Calculate $t_{ij}$ by solving the Linear program \eqref{eq:lp}.
     \STATE Calculate $a_{ij} = \frac{(t_{ij}-\delta_{ij})}{\epsilon}$  for $i,j=1$ to $N$
     \STATE Calculate $\K^i = \sum_{j=1}^N a_{ij} X^j$
 \end{algorithmic}
 \label{alg:OC}
\end{algorithm}

\section{Numerics}\label{sec:numerics}
This section contains results of numerical experiments where the
Algorithms~\ref{alg:const-gain}-\ref{alg:OC} are applied on the bimodal distribution problem
introduced in Example~\ref{ex:scalar}: The density $\rho$ is mixture
of two Gaussians $\mathcal{N}(-1,\sigma^2)$ and
$\mathcal{N}(+1,\sigma^2)$ with $\sigma^2=0.2$ and  and $h(x)=x$. The
exact solution is obtained using the explicit
formula~\eqref{eq:scalar} and is depicted in  \Fig{fig:truesol}. 

The following parameters are used in the numerical implementation of
the algorithms:
\begin{enumerate}
\item {\bf Galerkin:} Algorithm~\ref{alg:galerkin} with polynomial
  basis functions $\{x,x^2,\ldots,x^M\}$ for $M=1,3,5$. The case $M=1$
  gives the constant gain approximation (Algorithm~\ref{alg:const-gain}). 
\item {\bf Kernel:} Algorithm~\ref{alg:kernel} with $\epsilon=0.05,0.1,0.2$ and $L=1000$. 
\item {\bf Optimal coupling:} Algorithm~\ref{alg:OC} with $\epsilon=0.05,0.1,0.2$.  
\end{enumerate}

In the first numerical experiment, a fixed number of particles $N=100$
is drawn i.i.d from 
$\rho$. Figures~\ref{fig:gain-approx-bimodal}~(a)-(c)-(e) depicts the
approximate gain function obtained using the three algorithms.  For
the ease of comparison, the exact solution is also depicted.   

In the second numerical experiment, the empirical error is evaluated
as a function of the number of particles $N$.  For a single
simulation, the error is defined according to
\begin{equation}
\text{Error}:=\frac{1}{N}\sum_{i=1}^N | \K_{\text{alg}}(X^i) - \K_{\text{ex}}(X^i)|^2,
\label{eq:error}
\end{equation}
where $\{\X^i\}_{i=1}^N$ are the particles, $\K_{\text{alg}}$ is the
output of the algorithm and $\K_{\text{ex}}$ is the exact gain.  The
Monte-Carlo estimate of the error is evaluated by averaging over $1000$
simulations.  In each simulation, a new set of particles is sampled
which is used as an input consistently for the three
algorithms. Figure~\ref{fig:gain-approx-bimodal}~(b)-(d)-(f) depict
the Monte-Carlo estimate of the error as a function of the number of
particles.

In the third numerical experiment, the effect of varying the parameter
$\epsilon$ is investigated for the kernel-based and the optimal
coupling algorithms.  In this experiment, a fixed number of particles
$N=200$ is used.  Figure~\ref{fig:gain-approx-bimodal-eps} depict the
Monte-Carlo estimate of the error as a function of the parameter
$\epsilon$.

\begin{figure*}[t]
\centering
\begin{tabular}{cc}
\subfigure[Galerkin gain approxmiation for $N=100$]{
\includegraphics[width=0.8\columnwidth]{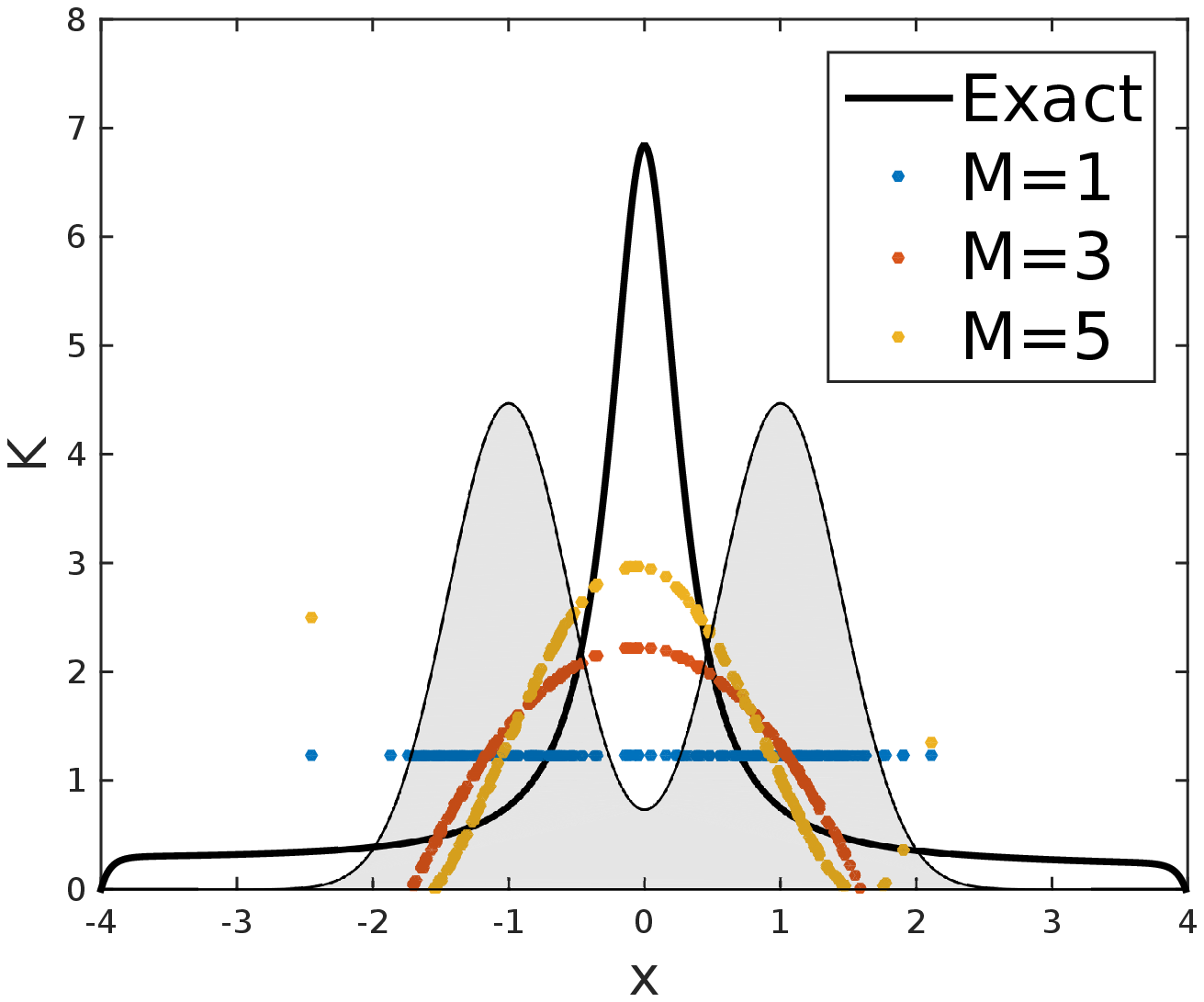}
} &
\subfigure[Galerkin gain approx. error as a function of $N$]{
\includegraphics[width=0.8\columnwidth]{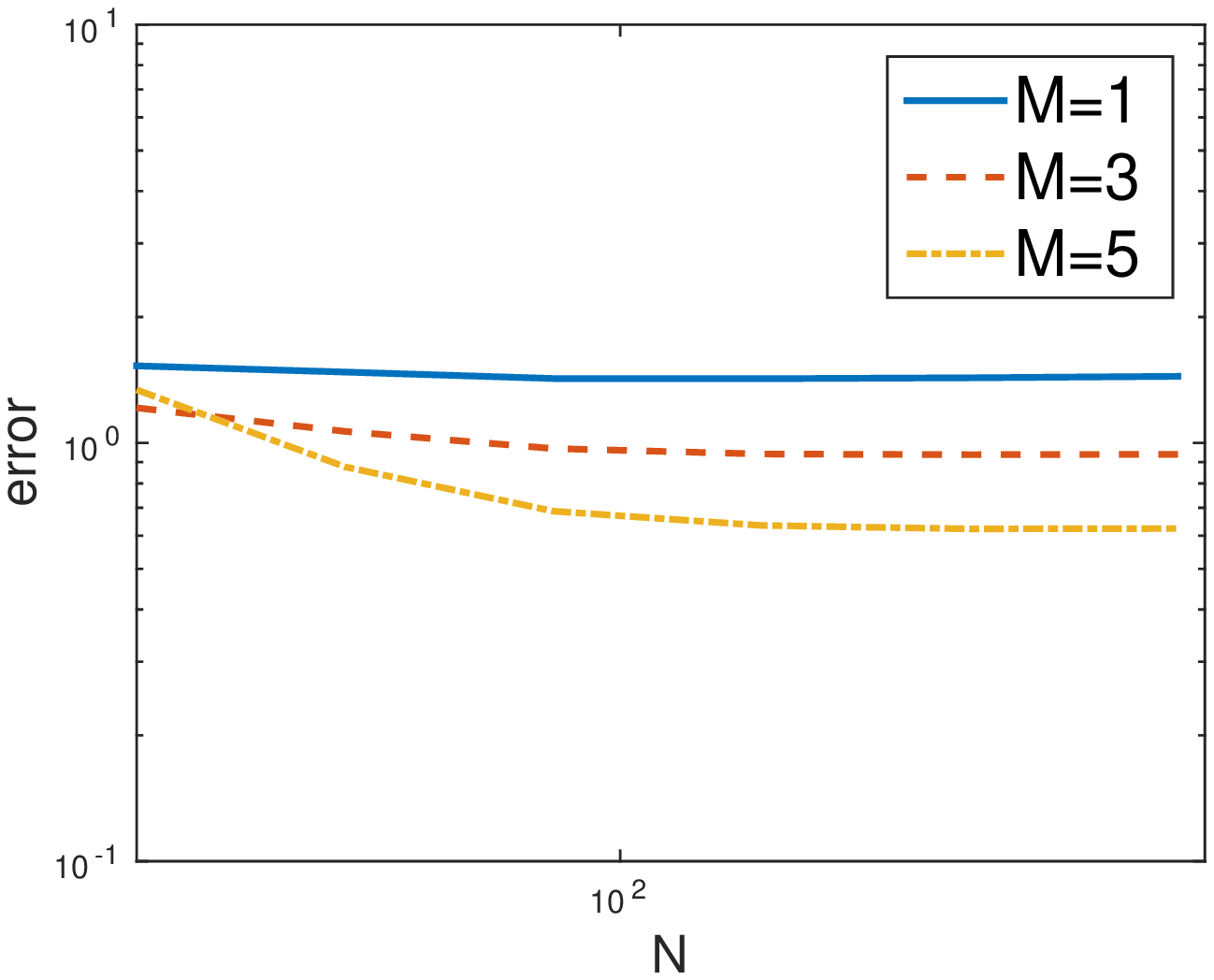}
} \\
\subfigure[Kernel-based gain approximation for $N=100$]{
\includegraphics[width=0.8\columnwidth]{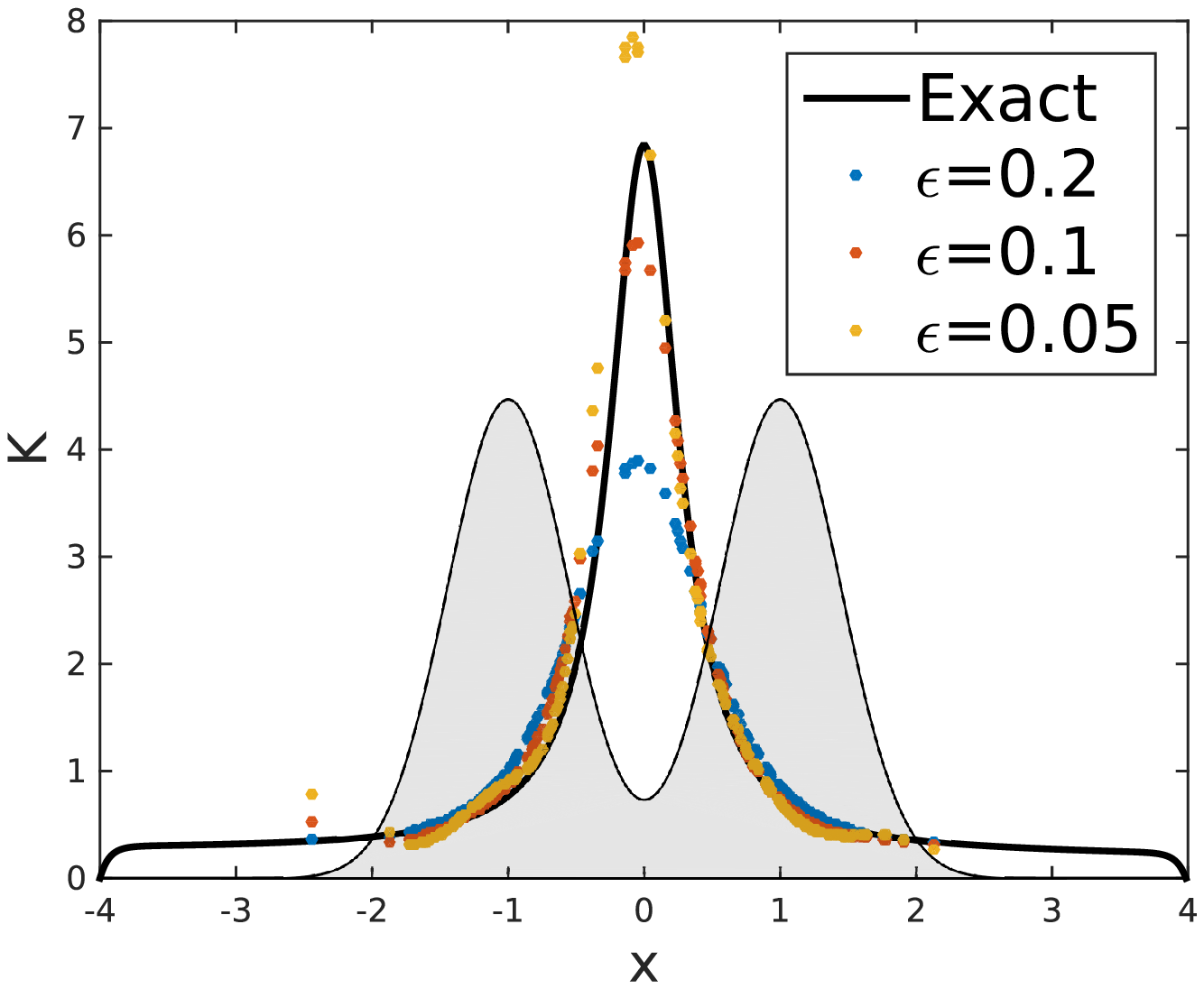}
} &
\subfigure[Kernel-based gain approx. error as a function of $N$]{
\includegraphics[width=0.8\columnwidth]{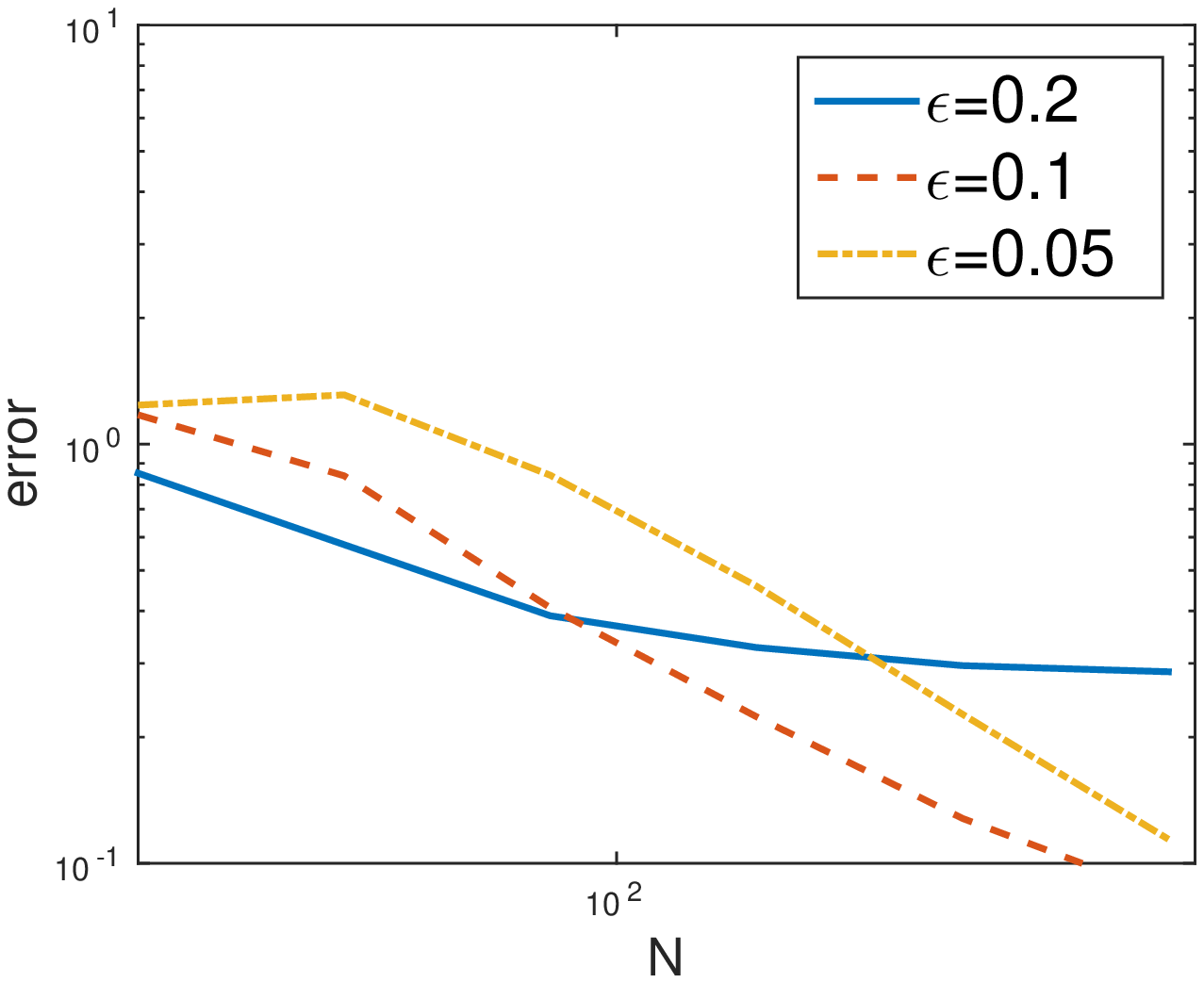}
} \\
\subfigure[Optimal Coupling gain approximation for $N=100$]{
\includegraphics[width=0.8\columnwidth]{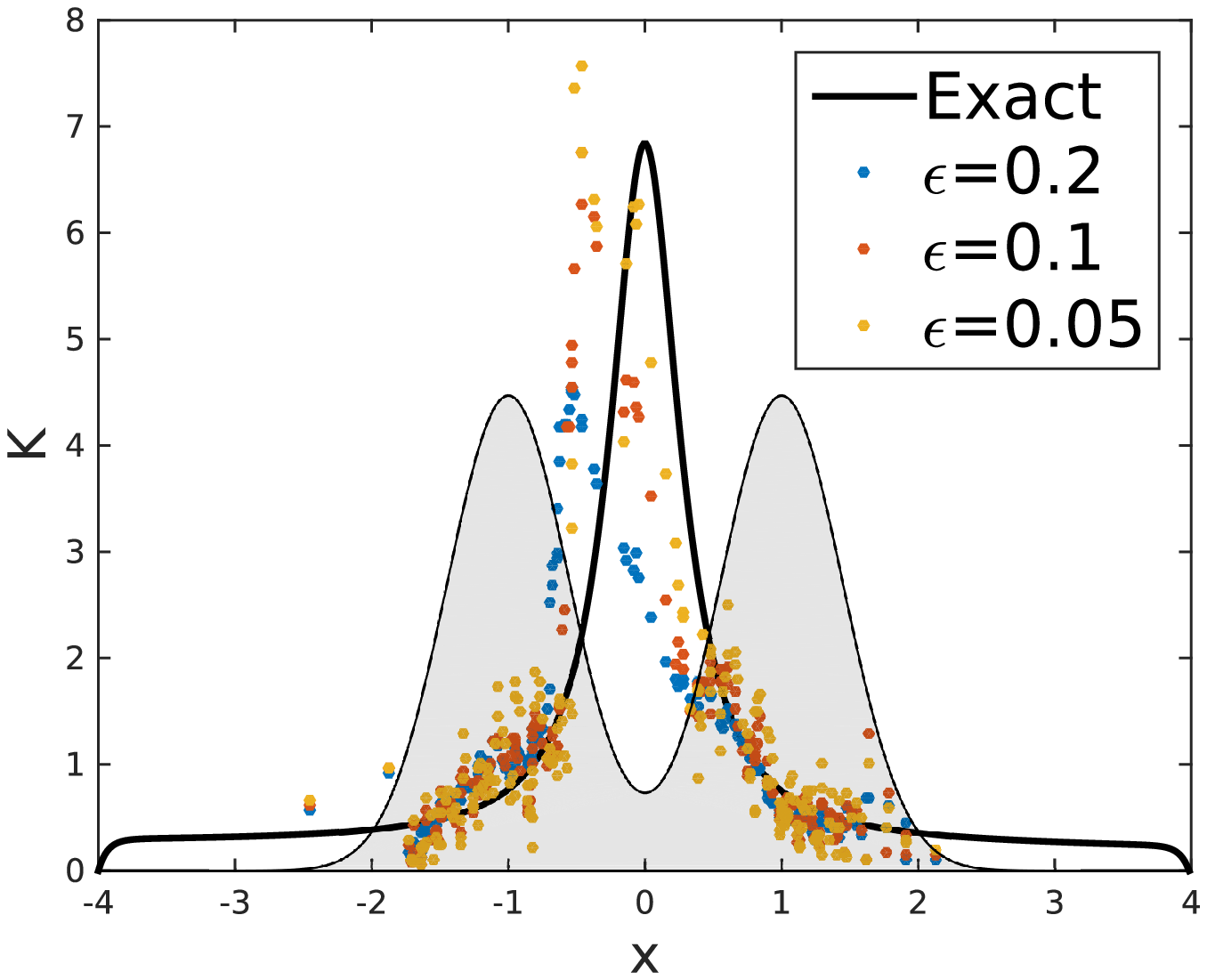}
} &
\subfigure[Optimal Coupling gain approx. error as a function of $N$]{
\includegraphics[width=0.8\columnwidth]{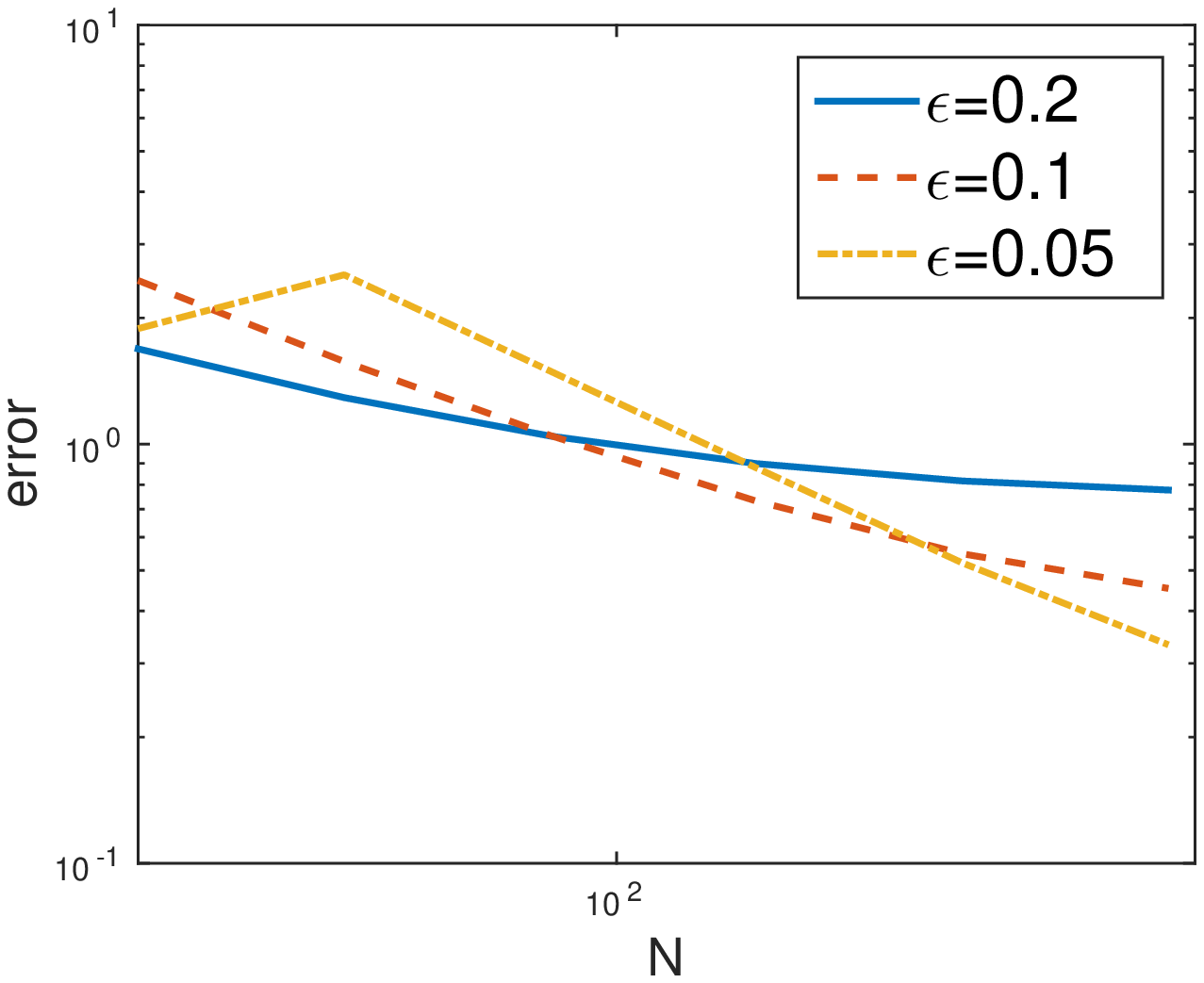}
} 
\end{tabular}
\caption{Comparison of the gain function approximations obtained using
  Galerkin (part~(a)), kernel (part~(c)), and the optimal coupling
  (part~(e)) algorithms.  The exact gain function is depicted as a
  solid line and the density $\rho$ is depicted as a shaded region
  in the background.  The parts~(b)-(d)-(e) depict the Monte-Carlo
  estimate of the empirical error~\eqref{eq:error} as a function of
  the number of particles $N$.  The Monte-Carlo estimate is obtained
  by averaging the empirical error over 100 simulations.}
\label{fig:gain-approx-bimodal}
\end{figure*}

The following observations are made based on the results of these
numerical experiments: 
\begin{enumerate}
\item (Figure \ref{fig:gain-approx-bimodal}~(a)-(b)) The accuracy of the Galerkin algorithm improves as the number of basis 
function increases.  For a fixed number of particles, the
matrix $A$ becomes poorly conditioned as the number of basis functions
becomes large. This can lead to numerical instabilities in solving the
matrix equation~\eqref{eq:Acb}.   

\item (Figure \ref{fig:gain-approx-bimodal}~(a)-(c)-(d)) The
  kernel-based and optimal coupling algorithms, preserve the
  positivity property of the exact gain. The positivity
  property of  the gain is not necesarily preserved in the Galerkin
  algorithm.  The correct sign of the gain is important in filtering
  applications as the gain determines the direction of drift of the
  particles.  A wrong sign can lead to divergence of the particle
  trajectories. 

\item (Figure \ref{fig:gain-approx-bimodal-eps}) For a fixed number of
  particles, there is an optimal value of $\epsilon$ that minimizes
  the error for the kernel-based and the optimal coupling
  algorithms. For the kernel-based algorithm, it
  is shown in~\cite{Amir_ACC17} that, for small $\epsilon$ and large
  $N$, the error scales as $O(\epsilon)
  + O(\frac{1}{\epsilon^{d/2+1}\sqrt{N}})$.  As
  $\epsilon\rightarrow\infty$, the approximate gain converges to the
  constant gain approximation.  In particular, the error remains
  bounded even for large values of $\epsilon$.

  \item The optimal coupling algorithm leads to spatially more irregular approximations which
  nevertheless converge as the number of particles increases. The optimal choice of
  the parameter $\epsilon$ depends on the particle size. Finally, a spatially more regular approximation
  could be obtained using kernel dressing, e.g., by convoluting the particle approximation with
  Gaussian kernels.  
\end{enumerate}

\begin{figure}[t]
\centering
\includegraphics[width=0.8\columnwidth]{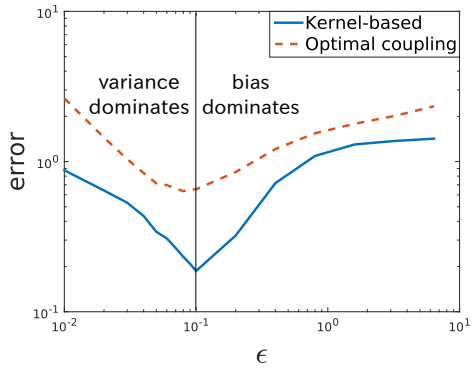}
\caption{Comparison of the Monte-Carlo estimate of the empirical
  error~\eqref{eq:error} as a function of the parameter $\epsilon$.  The number of particles $N=200$.}
\label{fig:gain-approx-bimodal-eps}
\end{figure}

\section{Conclusion}
\label{sec:conc}

We have summarized an interacting particle representation of the classic Kalman-Bucy filter and its extension
to nonlinear systems and non-Gaussian distributions under the general framework of FPFs. This framework is attractive
since it maintains key structural elements of the classic Kalman-Bucy filter, namely, a gain factor and the innovation.
In particular, the EnKF has become widely used in data assimilation for atmosphere-ocean dynamics and oil reservoir exploration. Robust extensions of
the EnKF to non-Gaussian distributions are urgently needed and FPFs provide a systematic approach for such extensions
in the spirit of Kalman's original work. However, interacting particle
representations come at a price; they require approximate solutions of
an elliptic PDE or a coupling problem, when viewed from a
probabilistic perspective. Hence, robust and efficient numerical
techniques for FPF-based interacting particle systems 
and study of their long-time behavior will be a primary focus of future research.

\section{Acknowledgement}

The first author AT was supported in part by the Computational Science
and Engineering (CSE) Fellowship at the University of Illinois at
Urbana-Champaign (UIUC).  
The research of AT and PGM at UIUC was
additionally supported by the National Science Foundation (NSF) grants 1334987 and 1462773.  This
support is gratefully acknowledged.
The research of JdW and SR has been partially funded by Deutsche
Forschungsgemeinschaft (DFG) through the grant CRC 1114 ``Scaling Cascades
in Complex Systems'', Project (A02) ``Multiscale data and asymptotic
model assimilation for atmospheric flows'' and the grant CRC 1294
``Data Assimilation'', Project (A02) ``Long-time stability and accuracy of
ensemble transform filter algorithms''.

\bibliographystyle{asmems4}
\bibliography{fpfbib}

\appendix
\section{Exactness of the Feedback Particle Filter}
\label{sec:theory}
The objective of this section is to describe the consistency result
for the feedback particle
algorithm~\eqref{eqn:particle_filter_nonlin_intro}, in the sense that
the posterior distribution of the particle exactly matches the
posterior distribution in the mean-field limit as $N \to \infty$.  To
put this in a mathematical framework, let $\pi^*_t$ denote the
conditional distribution of $X_t\in{\cal X}$ given history (filtration) of
observations $\mathcal{Z}_t:=\sigma\{Z_s,s\leq t\}$\footnote{The state
  space ${\cal X}$ is $\Re^d$ in this paper but the considerations of
  this section also apply when ${\cal X}$ is a differential
  manifold.}.  For a real-valued function $f:{\cal X}\to
\Re$ define the action of $\pi^*_t$ on $f$ according to:
\begin{equation*}
\pi^*_t(f):= \expect[f(X_t)|\clZ_t]
\end{equation*}
The time evolution of $\pi^*_t(f)$ is described by Kushner-Stratonovich pde (see Theorem 5.7 in \cite{xiong2008}),
\begin{equation}
\begin{aligned}
\pi^*_t(f) = &\pi^*_0(f) + \int_0^t\pi^*_s(\mathcal{L}f)\ud s \\&+\int_0^t \left(\pi^*_s(fh)-\pi^*_s(f)\pi^*_s(f)\right)\left(\ud Z_s - \pi_s^*(h)\ud s\right)
\end{aligned}
\label{eq:pi*t}
\end{equation}
where $f \in C^\infty_c({\cal X})$ (smooth functions with compact support), and 
\[\mathcal{L}f := a(x) \cdot \nabla f(x)  + \frac{1}{2}\sum_{k,l=1}^d \sigma_k(x)\sigma_l(x)\frac{\partial^2f}{\partial x_k \partial x_l}(x)\]

Next define $\pi_t$ to be the conditional distribution of $X^i_t\in
{\cal X}$ given $\clZ_t$.  It's action on real-valued function is
defined according to:
\begin{equation*}
\pi_t(f):= \expect[f(X^i_t)|\clZ_t]
\end{equation*}
The time evolution of $\pi_t(f)$ is described by the Fokker-Planck equation (see Proposition 1 in \cite{yang2016}):
\begin{equation}
\begin{aligned}
\pi_t(f) = &\pi_0(f) + \int_0^t\pi_s(\mathcal{L}f)\ud s + \int_0^t \pi_s(\k \cdot \nabla f)\ud Z_s \\&+ \int_0^t \pi_s(u \cdot \nabla f) \ud s+ \frac{1}{2}\int_0^t \pi_s(\k \cdot \nabla (\k \cdot \nabla f ))\ud s
\end{aligned}
\label{eq:pit}
\end{equation}
where $\k$ is the gain function and $u$ is the control function as follows:

\medskip

\def\grad{\text{grad}}
\newcommand{\lr}[2]{\langle #1, #2 \rangle}

\noindent {\em 1) Gain function:} 
Let $\phi \in H^1({\cal X};\pi_t)$\footnote{$H^1({\cal X};\rho)$ is
  the Sobolev space of functions on ${\cal X}$ that
  are square-integrable with respect to density $\rho$ and whose
  (weak) derivatives are also square-integrable with respect to
  density $\rho$.} be the solution of a (weak form of the) Poisson equation:
\begin{equation}
 \begin{aligned}
  & \pi_t \big( \nabla \phi \cdot \nabla \psi \big) = \pi_t \big( (h-\pi_t(h)) \psi \big)\\
  & \pi_t (\phi) = 0 ~~~~(\text{mean-zero})
 \end{aligned}
 \label{eq:BVP}
\end{equation}
for all $\psi\in H^1(\mathcal{X};\pi_t)$. The gain
function $\K(x, t) = \nabla \phi(x)$.

\medskip

\noindent {\em 2) Control function:} The function $
 u(x,t) = -\frac{1}{2}\, \K(x,t)\,\big(h(x)+\pi_t(h)\big)$.

The existence-uniqueness of the gain function $\k$ requires additional
assumptions on the distribution $\pi_t$ and the function $h$. 
\begin{romannum}
\item {\bf Assumption A1:} The probability distribution $\pi_t$ admits
  a spectral gap. That is , $\exists \,\lambda >0$ such that for all
  functions $f \in H^1_0({\cal X},\pi_t)$,
\begin{equation*}
\pi_t(|f|^2) \leq \frac{1}{\lambda}\pi_t(|\nabla f|^2)
\end{equation*}
for $t \in [0,T]$.
\item {\bf Assumption A2:} The function $h \in L^2({\cal X};\pi_t)$,
  the space of square-integrable functions with respect to $\pi_t$. 
\end{romannum}
Under the Assumptions (A1) and (A2) the Poisson equation~\eqref{eq:BVP} has a unique solution $\phi \in H^1_0(\mathcal{X},\pi_t)$ and the resulting control and gain function will be admissible~\cite{yang2016}. 
The consistency of the feedback particle filter is stated in the following Theorem. 
\begin{theorem}
Let $\pi^*_t$ and $\pi_t$ satisfy the forward equations~\eqref{eq:pi*t} and~\eqref{eq:pit}, respectively. Then assuming $\pi^*_0=\pi_0$, we have:
\begin{equation}
\pi^*_t(f)=\pi_t(f)
\end{equation}
for all $t \in [0,T]$ and all functions $f \in C^\infty_c(\Re^d)$.
\end{theorem}
\begin{proof}
Using~\eqref{eq:pi*t} and~\eqref{eq:pit} it is sufficient to show the following two identities:
\begin{align*}
  \pi_s(\k \cdot \nabla f) &=\pi_s(fh) - \pi_s(f)\pi_s(h)\\
  \pi_s(u\cdot \nabla f) + \frac{1}{2}\pi_s(\k \cdot \nabla (\k \cdot \nabla f ))&= -\left(\pi_s(fh) - \pi_s(f)\pi_s(h)\right)\pi_s(h) 
\end{align*} 
The first identity is obtained by using $\k=\nabla \phi$ and the weak form of the Poisson equation~\eqref{eq:BVP} with $\psi=f$. The second identity is obtained similarly. Use the expression for the control function $u$ to obtain,
\begin{align*}
\pi_s(u \cdot \nabla f)&= - \pi_s(\frac{h+\pi_s(h)}{2}\k \cdot \nabla f)\\
&=-\frac{1}{2}\pi_s((h-\pi_s(h))\k \cdot \nabla f) - \pi_s(h)\pi_s(\k \cdot \nabla f) \\
&= -\frac{1}{2}(\k \cdot \nabla (\k \cdot \nabla f))- \pi_s(h) \pi_s((h-\pi_s(h))f)
\end{align*} 
where in the last step the weak form of the Poisson equation~\eqref{eq:BVP} is used for the $\psi=\k \cdot \nabla f$ and $\psi=f$. This concludes the second identity and hence the Theorem. 
\end{proof}

\section{Relationship of gain function to the optimal transport map} 
\label{sec:lp_justification}

Let $S_{\epsilon}$ be the optimal transport map, solution of the problem~\eqref{eq:wass}. 
It is known that this map is of a gradient form~\cite{villani2003}. In the following we furthermore assume that the map $S_\epsilon=:\nabla \Phi_\epsilon$ is $C^1$ in the parameter $\epsilon$.  
For
any test function $\psi$
\[
C_\epsilon=\int \psi(\nabla \Phi_\epsilon(x))\pr(x)\ud x - \int \psi(x)\pr_\epsilon(x)\ud x
\equiv 0
\]
Therefore, 
\[
\left. \frac{\ud C_{\epsilon}}{\ud \epsilon} \right|_{\epsilon=0} =
\int \nabla \phi (x) \cdot \nabla \psi(x)\pr(x)\ud x - \int (h(x)-\hat{h}) \psi(x) \pr(x)\ud x = 0
\]
where $\phi := \frac{\ud \Phi_\epsilon}{\ud \epsilon}\big|_{\epsilon=0}$. This is the weak form of the Poisson equation~\eqref{eqn:EL_phi_intro}. 
Therefore, 
\begin{equation*}
\K = \nabla \phi = \frac{\ud S_\epsilon}{\ud \epsilon}\bigg|_{\epsilon=0}
\end{equation*}

\end{document}